\documentclass{article}[11pt]
\usepackage{amsthm,amssymb,amsmath}
\input xypic
\xyoption{all}
\newtheorem{theo}[subsubsection]{Theorem}
\newtheorem{lem}[subsubsection]{Lemma}
\newtheorem{prop}[subsubsection]{Proposition}

\newtheorem{exam}[subsubsection]{Example}
\newtheorem{rem}[subsubsection]{Remark}

\newcommand{\E}{{\cal E}}
\newcommand{\A}{{\cal A}}

\newcommand{\cC}{{\cal C}}

\newcommand{\cG}{{\cal G}}
\newcommand{\caH}{{\cal H}}
\newcommand{\cJ}{{\cal J}}
\newcommand{\cM}{{\cal M}}
\newcommand{\cN}{{\cal N}}

\newcommand{\End}{{\cal E}{\it nd}}
\newcommand{\Vect}{{\cal V}{\it ect}}
\newcommand{\sve}{{\scriptscriptstyle {\vee}}}
\begin{document}
\author{Alexei Davydov}
\title{Nuclei of categories with tensor products} \maketitle
\begin{center}
Department of Mathematics, Division of Information and
Communication Sciences, Macquarie University, Sydney, NSW 2109,
Australia
\end{center}
\begin{center}
davydov@math.mq.edu.au
\end{center}
\begin{abstract}
Following the analogy between algebras (monoids) and monoidal categories the construction of nucleus for non-associative algebras is simulated on the categorical level.

Nuclei of categories of modules are considered as an example.
\end{abstract}
\section{Introduction}
The property
\begin{equation}\label{peq}
xJ(y,z,w) - J(xy,z,w) + J(x,yz,w) - J(x,y,zw) + J(x,y,z)w = 0
\end{equation}
of the {\em associator} $J(x,y,z) = x(yz) - (xy)z$ guarantees that
the subspaces
\begin{equation}\label{nl}
N_l(A) = \{a\in A:\ J(a,x,y) = 0\ \forall x,y\in A\},
\end{equation}
\begin{equation}\label{nm}
N_m(A) = \{a\in A:\ J(x,a,y) = 0\ \forall x,y\in A\},
\end{equation}
\begin{equation}\label{nr}
N_r(A) = \{a\in A:\ J(x,y,a) = 0\ \forall x,y\in A\}
\end{equation}
of the algebra $A$ are associative subalgebras. We call these subalgebras {\em left, middle} and {\em right nucleus} respectively.

In section \ref{nuclsemi}  we categorify these construction by defining left, middle and right nuclei $\cN_l(\cG), \cN_m(\cG),  \cN_r(\cG) $ for a category $\cG$ with a tensor product ({\em magmoidal} category). The tensor product transports to the nuclei. Moreover there is a natural choice of associativity constraint making them {\em semigroupal} categories (the difference between semigroupal and monoidal is the lack of unit object).

The categorical analogs of the inclusions $N_l(A) ,N_m(A), N_r(A) \subset A$ are the tensor product preserving (forgetful) functors $\cN_l(\cG), \cN_m(\cG),  \cN_r(\cG) \to \cG$.

If an algebra $A$ was already associative, its nuclei will coincide with it. For a semigroupal category $\cG$ we only get semigroupal functors $\cG \to \cN_l(\cG), \cN_m(\cG),  \cN_r(\cG)$, which split the corresponding forgetful functor. Semigroupal structures on $\cG$ are in 1-to-1 correspondence with splittings of the forgetful functor.

If a (non-associative) algebra $A$  has a unit (an element $e\in A$ such that $ex = xe = x$ for any $x\in A$) then the unit belongs to all nuclei $N_\epsilon (A)$ making them associative unital. In section \ref{unit} we modify the constructions of nuclei for the case of a magmoidal category with unit. The modified constructions give monodal categories $\cN^1_l(\cG), \cN^1_m(\cG),  \cN^1_r(\cG)$.

Constructions of self-symmetries of categorical nuclei which we describe in section \ref{funcnuc} do not have any analogs in algebra. Here we associate, to any natural transformation of the tensor product functor on $\cG$, monoidal autoequivalences of the nuclei of $\cG$.

In section \ref{subnucl} we define certain subcategories of nuclei and show how they are acted upon by the self-symmetries of section \ref{funcnuc}.

Note that in the classical algebraic construction the multiplication map of an algebra $A$ makes it a left $N_l(A)$-module, a right $N_r(A)$-module, and a left- and right- $N_m(A)$-module. Moreover, left multiplications by elements of $N_l(A)$ commute with right multiplications by elements of $N_m(A)$ and $N_r(A)$; while right multiplications by elements of $N_r(A)$ commute with right multiplications by elements of $N_m(A)$ and $N_l(A)$. Finally, the multiplication map $A\otimes A\to A$ factors through $A\otimes_{N_m(A)}A$ and is a morphism of $N_l(A)$-$N_r(A)$-bimodules.
In section \ref{actnucl} we categorify some of these constructions by defining actions of nuclei $\cN(\cG)$ on the category $\cG$.

This could give the impression that coherence laws of monoidal categories guarantee that categorifications of algebraic constructions are nicely coherent themselves. Unfortunately it is not always true, especially if commutativity enters the picture. For example, if $A$ is a commutative (non-associative) algebra, then different nuclei are related as follows $$N_l(A) = N_r(A) \subset N_m(A).$$ Indeed, the following chain of equalities proves that an element $a$ of $N_l(A)$ belongs to $N_m(A)$: $$x(ay) = (ay)x = a(yx) = a(xy) = (ax)y = (xa)y.$$ To see that $a$ belongs to $N_r(A)$ one needs a slightly longer chain: $$x(ya) = x(ay) = ... = (xa)y = (ax)y = a(xy) = (xy)a.$$
On the level of categories we have the following relations between nuclei of a category $\cG$ with commutative tensor product: $$\cN_l(\cG) \simeq \cN_r(\cG) \to \cN_m(\cG).$$ But the functors between different nuclei are not monoidal.

In the second part of the paper we deal with a similar construction making maps between associative algebras multiplicative.
Let $f:A\to B$ be a linear map of associative algebras. Define $m_f(x,y) = f(xy) - f(x)f(y)$. The property
$$f(x)m_f(y,z) - m_f(xy,z) + m_f(x,yz) - m_f(x,y)f(z) = 0$$ implies that the subspaces of $A$ $$M_l(f) = \{a\in
A:\ m_f(a,x) = 0\ \forall x\in A\},\quad M_r(f) = \{a\in A:\ m_f(x,a) = 0\ \forall x\in A\}$$ are subalgebras (we call them the left and right {\em multiplicants} of the map $f$).
It is straightforward to see that the restrictions of $f$ to $M_l(f)$ and $M_r(f)$ are homomorphisms of algebras.

Again we face problems when we try to bring  in commutativity. For example, the left and right multiplicants $M_l(f), M_r(f)$ of a linear map $f:A\to B$ between commutative algebras coincide. The following chain of equalities shows that any $a\in M_l(f)$ belongs to $M_r(f)$: $$f(xa) = f(ax) = f(a)f(x) = f(x)f(a).$$ But the equivalence $\cM_l(F),\to \cM_r(F)$ between multiplicants of a functor $F:\cG\to \caH$ between symmetric monoidal categories fails to be monoidal.

\section*{Acknowledgment}
 The work on the paper was supported by
Australian Research Council grant DP00663514. The author thanks members of the Australian Category Seminar for stimulating discussions.
Special thanks are due to R. Street for invaluable support during the work on the paper.

\section{Nuclei of a category with a tensor product}
\subsection{Nuclei as semigroupal categories}\label{nuclsemi}
Let $\cG$ be a category with a tensor product functor $\otimes:\cG\times\cG\to\cG$ without any associativity
condition ({\em magmoidal} category). Following (\ref{nl}) define the {\em left nucleus} $\cN_l(\cG)$ of $\cG$ as the category of pairs
$(A,a)$, where $A\in\cG$ and $a$ denotes a family of isomorphisms $$a_{X,Y}:A\otimes(X\otimes Y)\to(A\otimes
X)\otimes Y$$ natural in $X,Y$. We will consider a morphism $f:A\to B$ from $\cG$ as a morphism $(A,a)\to(B,b)$ in
$\cN_l(\cG)$ if the diagram

\xymatrix{ A\otimes(X\otimes Y) \ar[d]^{f\otimes(X\otimes Y)}\ar[r]^{a_{X,Y}} & (A\otimes X)\otimes Y
\ar[d]^{(f\otimes X)\otimes Y}
\\
B\otimes(X\otimes Y) \ar[r]^{b_{X,Y}} & (B\otimes X)\otimes Y }

is commutative for any $X,Y\in\cG$. Introduce the tensor product on the category $\cN_l(\cG)$ by
$(A,a)\otimes(B,b) = (A\otimes B,a|b)$, where $a|b$ is defined by the pentagon diagram:

$$
\xygraph{ !{0;/r4.5pc/:;/u4.5pc/::}[]*+{A\otimes(B\otimes(X\otimes Y))} (
  :[u(1.1)r(1.7)]*+{(A\otimes B)\otimes(X\otimes Y)} ^{a_{B,X\otimes Y}}
  :[d(1.1)r(1.7)]*+{((A\otimes B)\otimes X)\otimes Y}="r" ^{(a|b)_{X,Y}}
  ,
  :[r(.6)d(1.5)]*+!R(.3){A\otimes((B\otimes X)\otimes Y)} ^{A\otimes b_{X,Y}}
  :[r(2.2)]*+!L(.3){(A\otimes(B\otimes X))\otimes Y} ^{a_{B\otimes X,Y}}
  : "r" ^{a_{B,X}\otimes Y}
)
}$$

The following diagram implies that for two morphisms $f:(A,a)\to(B,b), g:(A',a')\to(B',b')$ the tensor product
$f\otimes g$ is a morphism between tensor products of pairs $(A,a)\otimes(A',a')\to(B,b)\otimes(B',b')$.

\begin{equation}\label{tpm}
\xygraph{ !{0;/r6pc/:;/u6pc/::}[]*+{A(B(XY))}="tl" (
  :[u(.7)rr]*+{(AB)(XY)} ^{a_{B,XY}}
  ( :[d]*+{(A'B')(XY)}="t" _{(fg)(XY)}
  )
  :[d(.7)rr]*+{((AB)X)Y}="tr" ^{(a|b)_{X,Y}}
  ( :[d]*+{((A'B')X)Y}="r" ^{((fg)X)Y}
  )
  ,
  :[r(.8)d]*+{A((BX)Y)}="tlb" ^(.4){Ab_{X,Y}}
  ( :[d]*+{A'((B'X)Y)}="lb" ^{f((gX)Y)}
  )
  :[r(2.4)]*+{(A(BX))Y}="trb" ^{a_{BX,Y}}
  ( :[d]*+{(A'(B'X))Y}="rb" _{(f(gX))Y}
  )
  : "tr" ^(.6){a_{B,X}Y}
  ,
  :[d]*+{A'(B'(XY))} _{f(g(XY))}
  ( : "t" ^(.6){a'_{B',XY}} |!{"tl";"tlb"}\hole
    : "r" ^(.4){(a'|b')_{X,Y}} |!{"trb";"tr"}\hole
    ,
    : "lb" _{A'b'_{X,Y}}
    : "rb" _{a'_{B'X,Y}}
    : "r" _{a'_{B',X}Y}
  )
)
}
\end{equation}

In particular, the forgetful functor $\cN_l(\cG)\to\cG$ sending $(A,a)$ to $A$ is magmoidal.

It follows from the diagram below that for $(A,a),(B,b),(C,c)\in\cN_l(\cG)$ the isomorphism
$a_{B,C}:A\otimes(B\otimes C)\to(A\otimes B)\otimes C$ belongs to the nucleus $\cN_l(\cG)$. Indeed, it is equivalent to the commutativity of the square

\xymatrix{ (A\otimes(B\otimes C))\otimes(X\otimes Y) \ar[d]^{a_{B,C}\otimes(X\otimes Y)}\ar[rr]^{(a|(b|c))_{X,Y}} & &
((A\otimes(B\otimes C))\otimes X)\otimes Y \ar[d]^{(a_{B,C}\otimes X)\otimes Y}
\\
((A\otimes B)\otimes C)\otimes(X\otimes Y) \ar[rr]^{((a|b)|c)_{X,Y}} & & (((A\otimes B)\otimes C)\otimes X)\otimes Y
}

which is one of the faces of the diagram ($K_5$ in Stasheff's notations \cite{st}):

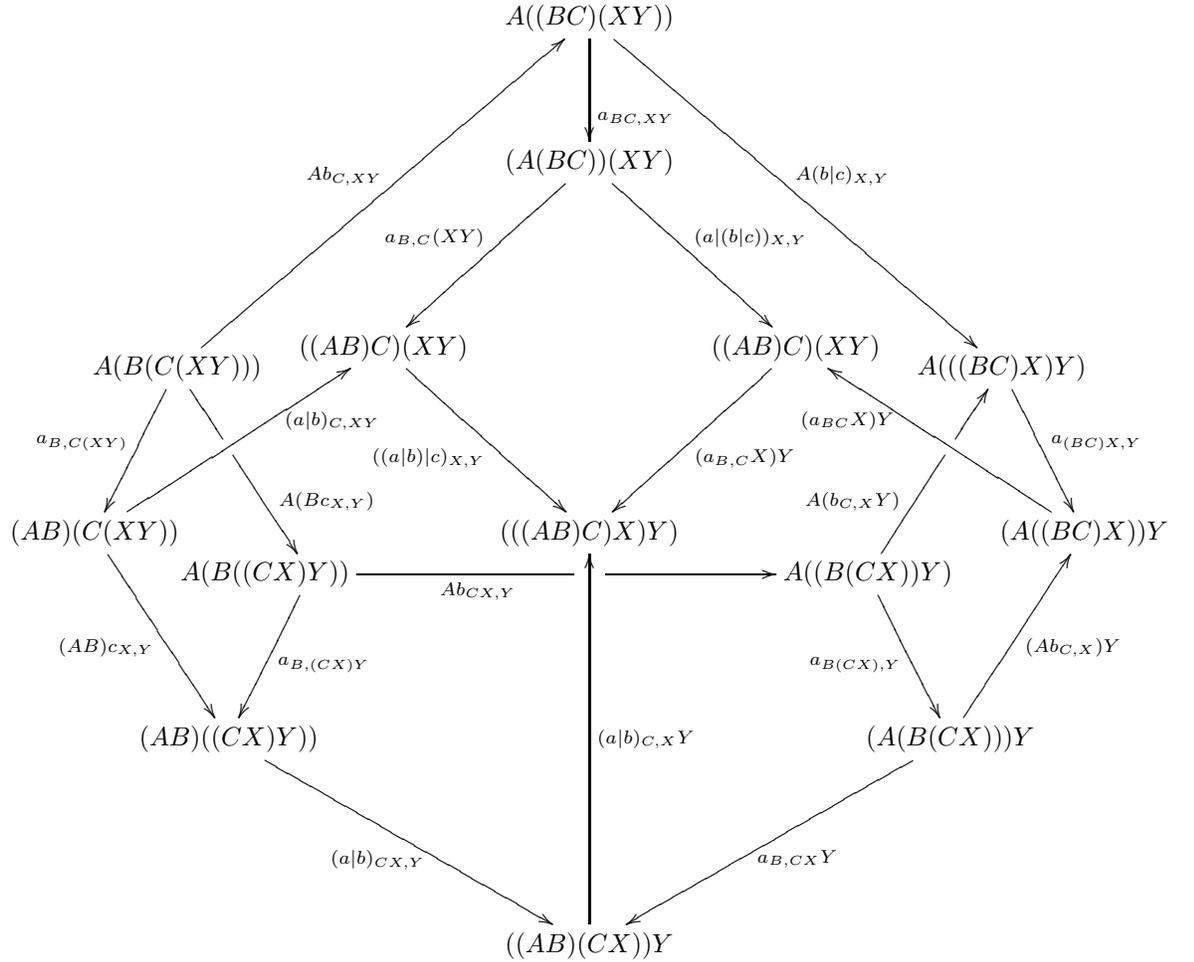
$$
\xygraph{ !{0;/r6.5pc/:;/u6.5pc/::} []*+{A(B(C(XY)))} ( :[u(1.7)rr]*+{A((BC)(XY))} ^{Ab_{C,XY}}
  (
    :[d(.7)]*+{(A(BC))(XY)} ^(.7){a_{BC,XY}}
    (
      :[d(.9)l]*+{((AB)C)(XY)}="lt" _{a_{B,C}(XY)}
      :[d(.9)r]*+{(((AB)C)X)Y)}="tb" _{((a|b)|c)_{X,Y}}
    ,
      :[d(.9)r]*+{((AB)C)(XY)}="rt" ^{(a|(b|c))_{X,Y}}
      :"tb" ^{(a_{B,C}X)Y}
    )
  , :[d(1.7)rr]*+{A(((BC)X)Y)}="rm" ^{A(b|c)_{X,Y}}
    :[r(.4)d(.8)]*+{(A((BC)X))Y}="rb" ^{a_{(BC)X,Y}}
    : "rt" ^(.7){(a_{BC}X)Y}
  )
,
  :[l(.4)d(.8)]*+{(AB)(C(XY))}="ll" _{a_{B,C(XY)}}(: "lt" _(.7){(a|b)_{C,XY}})
  :[r(.65)d(1)]*+{(AB)((CX)Y))}="lb" _{(AB)c_{X,Y}}
  :[r(1.75)d]*+{((AB)(CX))Y}="bb" _{(a|b)_{CX,Y}}
  : "tb" _{(a|b)_{C,X}Y}
,
  :[d(1)r(.65)]*+!R(.5){A(B((CX)Y))}="lm" |!{"ll";"lt"}\hole ^(.7){A(Bc_{X,Y})}
  ( : "lb" ^{a_{B,(CX)Y}}
  ,
    :[r(2.7)]*+{A((B(CX))Y)} _(.3){Ab_{CX,Y}}
        |*+{\hole}
    ( :"rm" ^(.3){A(b_{C,X}Y)} |!{"rb";"rt"}\hole
    , : [d(.8)r(.4)]*+{(A(B(CX)))Y} _{a_{B(CX),Y}}
      ( : "rb" _{(Ab_{C,X})Y} , : "bb" ^{a_{B,CX}Y} )
    )
  )
)
}$$

Two other square faces of this diagram are commutative by the naturality while the pentagon faces commute by the definition of the tensor product in $\cN_l(\cG)$. Thus
$$\phi_{(A,a),(B,b),(C,c)}:(A,a)\otimes((B,b)\otimes(C,c))\to ((A,a)\otimes(B,b))\otimes(C,c)$$ defined by $\phi_{(A,a),(B,b),(C,c)} = a_{B,C}$, is a morphism in the category $\cN_l(\cG)$.

\begin{theo}\label{semigr}
The category $\cN_l(\cG)$ with tensor product and associativity constraint defined above is semigroupal.
\end{theo}
\begin{proof}
All we need to verify is that the associativity constraint satisfies the pentagon axion, which says that the diagram

$$
\xygraph{ !{0;/r4.5pc/:;/u4.5pc/::}[]*+{A\otimes(B\otimes(C\otimes D))} (
  :[u(1.1)r(1.7)]*+{(A\otimes B)\otimes(C\otimes D)} ^{a_{B,C\otimes D}}
  :[d(1.1)r(1.7)]*+{((A\otimes B)\otimes C)\otimes D}="r" ^{(a|b)_{C,D}}
  ,
  :[r(.6)d(1.5)]*+!R(.3){A\otimes((B\otimes C)\otimes D)} ^{A\otimes b_{C,D}}
  :[r(2.2)]*+!L(.3){(A\otimes(B\otimes C))\otimes D} ^{a_{B\otimes C,D}}
  : "r" ^{a_{B,C}\otimes D}
)
}$$
is commutative for any $(A,a),(B,b),(C,c),(D,d)\in\cN_l(\cG)$. This obviously follows from the definition of $a|b$.
\end{proof}

Semigroupal structures on the category $\cG$ correspond to functors $\cG\to \cN_l(\cG)$ which split the forgetful
functor $\cN_l(\cG)\to \cG$.
\begin{theo}
If $\cG$ is a semigroupal category with the associativity constraint $\phi$, the assignment $A\mapsto (A,\phi_{A,-,-})$
defines a semigroupal functor $\cG\to \cN_l(\cG)$ whose composition $\cG\to \cN_l(\cG)\to \cG$ with the forgetful
functor is the identity.

Conversely, a functor $\cG\to \cN_l(\cG)$, whose composition $\cG\to \cN_l(\cG)\to \cG$ with the forgetful functor is
the identity, defines an associativity constraint on $\cG$ thus turning it into a semigroupal category.
\end{theo}
\begin{proof}
The pentagon axiom for $\phi$ implies that $\phi_{A,-,-}|\phi_{B,-,-} = \phi_{A\otimes B,-,-}$, so that $A\mapsto
(A,\phi_{A,-,-})$ is a strict semigroupal functor.

A splitting $\cG\to \cN_l(\cG)$ of the forgetful functor $\cN_l(\cG)\to \cG$ must have the form: $A\mapsto (A,a)$. It is straightforward to see that $\phi_{A,X,Y} = a_{X,Y}$ is an associativity constraint on $\cG$.
\end{proof}

\begin{rem}
\end{rem}
Diagrams of this section (as well as their faces, edges and vertices) can be parameterized by certain partially
parenthesised  words of objects (morphisms) of $\cG$. For example, the tensor product of two objects $X\otimes Y$
corresponds to the word $XY$, tensor products $A\otimes(X\otimes Y)$ and $(A\otimes X)\otimes Y$ correspond to the
words $A(XY)$ and $(AX)Y$ respectfully. The word $AXY$ corresponds to the associativity morphism $a_{X,Y}$ and the word
$ABXY$ to the pentagon diagram which was used to define the tensor product in $\cN_l(\cG)$. Five edges of this pentagon
correspond to five ways of placing a pair of brackets $(\ )$ in the word $ABXY$: $$A(BXY),\ (AB)XY,\ A(BX)Y,\ AB(XY),\
(ABX)Y$$ which represent $$A\otimes b_{X,Y},\ (a|b)_{X,Y},\ a_{B\otimes X,Y},\ a_{B,X\otimes Y},\ a_{B,X}\otimes
Y$$ respectively. The diagram for tensor product of morphisms (\ref{tpm}) corresponds to the word $fgXY$. Finally, Stasheff's
polytope corresponds to the word $ABCXY$. Again its faces are parameterized by partial parenthesisings of the word $ABCXY$
with a single pair of brackets $(\ )$. Six pentagon faces of the Stasheff polytope correspond to the words: $$A(BCXY),\
(AB)CXY,\ A(BC)XY,\ AB(CX)Y,\ ABC(XY),\ (ABCX)Y,$$ while three square faces are labeled by: $$(ABC)XY,\ A(BCX)Y,\
AB(CXY).$$ The fact that the associativity morphism $a_{B,C}$ is a morphism in the nucleus $\cN_l(\cG)$ follows from
the commutativity of the face $(ABC)XY$. Finally, the pentagon axiom for its associativity is encoded by $ABCD$.

Using the word presentation for diagram we can sketch the proof of the analog of the theorem \ref{semigr} for $\cN_m(\cG)$.
Define $\cN_m(\cG)$ as the category of pairs $(A,a)$, where $A\in\cG$ and $a$ is a natural collection of isomorphisms
$$a_{X,Y}:X\otimes(A\otimes Y)\to(X\otimes A)\otimes Y$$ which we label by the word $XAY$. As before morphisms in
$\cN_m(\cG)$ are morphisms in $\cG$ compatible with natural isomorphisms. We can define the tensor product of pairs
$(A,a)\otimes (B,a) = (A\otimes B,a|b)$ using the pentagon with the label $XABY$. The diagram of the shape $XfgY$
implies that the tensor product of morphisms $f\otimes g$ is a morphism of tensor products. It follows from
commutativity properties of the shape $XABCY$ that the formula $\phi_{(A,a),(B,b),(C,c)} = b_{A,C}$ defines a natural collection of isomorphisms in $\cN_m(\cG)$:
$$\phi_{(A,a),(B,b),(C,c)}:(A,a)\otimes((B,b)\otimes(C,c))\to ((A,a)\otimes(B,b))\otimes(C,c).$$ The word $ABCD$
gives the pentagon axiom for this constraint.

Similarly, define $\cN_r(\cG)$ as the category of pairs $(A,a)$, with $$a_{X,Y}:X\otimes(Y\otimes A)\to(X\otimes
Y)\otimes A$$ which we label by the word $XYA$. Again, morphisms in $\cN_r(\cG)$ are morphisms in $\cG$ compatible with
natural isomorphisms. The tensor product of pairs $(A,a)\otimes (B,a) = (A\otimes B,a|b)$ is defined by the pentagon
with the label $XYAB$. The diagram of the shape $XYfg$ implies that the tensor product of morphisms $f\otimes g$ is a
morphism of tensor products.The shape $XYABC$ implies that the formula $\phi_{(A,a),(B,b),(C,c)} = c_{A,B}$ defines a natural collection of isomorphisms in $\cN_r(\cG)$:
$$\phi_{(A,a),(B,b),(C,c)}:(A,a)\otimes((B,b)\otimes(C,c))\to ((A,a)\otimes(B,b))\otimes(C,c).$$ Finally, the word
$ABCD$ gives the pentagon axiom for this constraint.

\subsection{Units}\label{unit}

In this section
we slightly modify the construction to incorporate units. First we need to modify the setup.

Let now $\cG$ be a magmoidal category with {\em unit} (an object $1\in \cG$ with natural isomorphisms $l_X:1\otimes X\to X,\ r_X:X\otimes 1\to X$ such that $l_1=r_1$). Define the left {\em unital nucleus} $\cN^1_l(\cG)$ as the category of pairs $(A,a)$ where as before $a_{X,Y}:A\otimes(X\otimes Y)\to (A\otimes X)\otimes Y$ is natural collection of isomorphism but now satisfying {\em normalisation} conditions:
the diagrams

\xymatrix{A\otimes(1\otimes X) \ar[rr]^{a_{1,X}} \ar[rd]_{A\otimes l_X} & & (A\otimes 1)\otimes X \ar[dl]^{r_A\otimes X}  \\ & A\otimes X  }
\xymatrix{A\otimes(X\otimes 1) \ar[rr]^{a_{X,1}} \ar[rd]_{A\otimes r_X} & & (A\otimes X)\otimes 1 \ar[dl]^{r_{A\otimes X}} \\ & A\otimes X }
commute for any $X\in\cG$.

Define a natural collection of isomorphisms $i_{X,Y}:1\otimes(X\otimes Y)\to (1\otimes X)\otimes Y$ as the composition \xymatrix{1\otimes(X\otimes Y) \ar[r]^{l_{X\otimes Y}} & X\otimes Y \ar[r]^{(l_X\otimes Y)^{-1}} & (1\otimes X)\otimes Y .}

\begin{lem}\label{uobj}
The pair $(1,i)$ is an object of $\cN^1_l(\cG)$.
\end{lem}
\begin{proof}
The normalisation conditions for the collection $i$ follow from the commutative diagrams:
\newline
\xymatrix{ 1\otimes(1\otimes Y) \ar[rr]^{i_{1,Y}} \ar@/^2pt/[rd]^{l_{1\otimes Y}} \ar@/_20pt/[rd]_{1\otimes l_Y} & & (1\otimes 1)\otimes Y \ar@/^20pt/[ld]^{r_1\otimes Y} \ar@/_2pt/[ld]_{l_1\otimes Y} \\ & 1\otimes Y }

\xymatrix{1\otimes(X\otimes 1) \ar[rr]^{i_{,1}} \ar[rd]_{1\otimes r_X} \ar[rddd]_{l_{X\otimes 1}} & & (1\otimes X)\otimes 1 \ar[ld]^{r_{1\otimes X}} \ar[lddd]^{l_X\otimes 1} \\ & 1\otimes X \ar[d]^{l_X} \\ & X \\ & X\otimes 1 \ar[u]_>>>{r_X} }
Here the equation $l_{1\otimes Y} = 1\otimes l_Y$ (used in the first diagram) follows from the naturality of $l$: $\quad$
\xymatrix{1\otimes (1\otimes Y) \ar[r]^{l_{1\otimes Y}} \ar[d]_{1\otimes l_Y} & 1\otimes Y \ar[d]^{l_Y} \\ 1\otimes Y \ar[r]^{l_Y} & Y}

In the second diagram we use the naturality of $l$ with respect to $r_X$ and the naturality of $r$ with respect to $l_X$.
\end{proof}

\begin{theo}
The category $\cN^1_l(\cG)$ is monoidal with the unit object $(1,i)$ and unit natural isomorphisms
\begin{equation}\label{ui}
l_{(A,a)}:(1,i)\otimes(A,a)\to (A,a),\quad r_{(A,a)}:(A,a)\otimes(1,i)\to (A,a)
\end{equation}
defined by $l_{(A,a)} = l_A,\ r_{(A,a)} = r_A$.
\end{theo}
\begin{proof}
First we need to check that the tensor product of pairs as it was defined for $\cN_l(\cG)$ preserves the normalisation conditions. This follow from the commutative diagrams:
\newline
$$
\xygraph{ !{0;/r6.5pc/:;/u6.5pc/::}[]*+{A(B(1Y))} (
  :[r(1.1)]*+{A(BY)}="ml"_{A(Bl_Y)}
  :[r(1.2)]*+{(AB)Y}="mr"_{a_{B,Y}}
  ,
  :[u(1.1)r(1.7)]*+{(AB)(1Y)} ^{a_{B,1Y}}
  (
  :"mr"_{(AB)l_Y}
  ,
  :[d(1.1)r(1.7)]*+{((AB)1)Y}="r" ^{(a|b)_{1,Y}}
  :"mr"^{r_{AB}Y}
  )
  ,
  :[r(.6)d(1.1)]*+!R(.3){A((B1)Y)} _{Ab_{1,Y}}
  (
  :"ml"_{A(r_bY)}
  ,
  :[r(2.2)]*+!L(.3){(A(B1))Y} ^{a_{B1,Y}}
  (
  : "r" _{a_{B,1}Y}
  ,
  :"mr"^{(Ar_B)Y}
  )
  )
)
}$$
$$
\xygraph{ !{0;/r6.5pc/:;/u6.5pc/::}[]*+{A(B(X1))} (
  :[r(1.1)]*+{A(BX)}="ml"^{A(Br_X)}
  :[r(1.2)]*+{(AB)X}="mr"^{a_{B,X}}
  ,
  :[u(1.1)r(1.7)]*+{(AB)(X1)} ^{a_{B,X1}}
  (
  :"mr"_{(ABr_X}
  ,
  :[d(1.1)r(1.7)]*+{((AB)X)1}="r" ^{(a|b)_{X,1}}
  :"mr"^{r_{(AB)X}}
  )
  ,
  :[r(.6)d(1.1)]*+!R(.3){A((BX)1)} _{Ab_{X,1}}
  (
  :"ml"_{Ar_{BX}}
  ,
  :[r(2.2)]*+!L(.3){(A(BX))1} ^{a_{BX,1}}
  (
  : "r" _{a_{B,X}1}
  ,
  :"ml"_{r_{A(BX)}}
  )
  )
)
}$$
Thus the subcategory $\cN^1_l(\cG)$ is closed under tensor product in $\cN_l(\cG)$ and is a semigroupal category. By lemma \ref{uobj},  the pair $(1,i)$ is an object of $\cN^1_l(\cG)$. We need to check that the unit isomorphisms (\ref{ui}) are morphisms in he category $\cN^1_l(\cG)$. It is equivalent to the commutativity of the squares:

\xymatrix{(1A)(XY) \ar[r]^{(i|a)_{X,Y}} \ar[d]_{l_AXY} & ((1A)X)Y \ar[d]^{(l_AX)Y} &  (A1)(XY) \ar[r]^{(a|i)_{X,Y}} \ar[d]_{r_AXY} & ((A1)X)Y \ar[d]^{(r_AX)Y} \\ A(XY) \ar[r]^{a_{X,Y}} & (AX)Y & A(XY) \ar[r]^{a_{X,Y}} & (AX)Y}

which follows from the commutative diagrams:
$$
\xygraph{ !{0;/r6.5pc/:;/u6.5pc/::}[]*+{1(A(XY))} (
  :[r(1.1)]*+{A(XY)}="ml"_{l_{A(XY)}}
  :[r(1.2)]*+{(AX)Y}="mr"_{a_{X,Y}}
  ,
  :[u(1.1)r(1.7)]*+{(1A)(XY)} ^{i_{A,XY}}
  (
  :"ml"^{l_A(XY)}
  ,
  :[d(1.1)r(1.7)]*+{((1A)X)Y}="r" ^{(i|a)_{X,Y}}
  :"mr"^{(l_AX)Y}
  )
  ,
  :[r(.6)d(1.1)]*+!R(.3){1((AX)Y)} _{1a_{X,Y}}
  (
  :"mr"^{l_{(AX)Y}}
  ,
  :[r(2.2)]*+!L(.3){(1(AX))Y} ^{i_{AX,Y}}
  (
  : "r" _{i_{A,X}Y}
  ,
  :"mr"^{l_{AX}Y}
  )
  )
)
}$$
$$
\xygraph{ !{0;/r6.5pc/:;/u6.5pc/::}[]*+{A(1(XY))} (
  :[r(1.1)]*+{A(XY)}="ml"_{Al_{XY}}
  :[r(1.2)]*+{(AX)Y}="mr"_{a_{X,Y}}
  ,
  :[u(1.1)r(1.7)]*+{(A1)(XY)} ^{a_{1,XY}}
  (
  :"ml"^{r_A(XY)}
  ,
  :[d(1.1)r(1.7)]*+{((A1)X)Y}="r" ^{(a|i)_{X,Y}}
  :"mr"^{(r_AX)Y}
  )
  ,
  :[r(.6)d(1.1)]*+!R(.3){A((1X)Y)} _{Ai_{X,Y}}
  (
  :"ml"_{A(l_XY)}
  ,
  :[r(2.2)]*+!L(.3){(A(1X))Y} ^{a_{1X,Y}}
  (
  : "r" _{a_{1,X}Y}
  ,
  :"mr"^{(Al_X)Y}
  )
  )
)
}$$

Finally, the coherence for the unit isomorphisms

\xymatrix{(A,a)\otimes((1,i)\otimes(B,b)) \ar[rd]_{(A,a)\otimes l_{(B,a)}} \ar[rr]^{\phi_{(A,a),(1,i),(B,b)}} & &
((A,a)\otimes(1,i))\otimes(B,b) \ar[ld]^{r_{(A,a)}\otimes(B,b)} \\ & (A,a)\otimes(B,b) }

boils down to \xymatrix{A\otimes(1\otimes B) \ar[rr]^{a_{1,B}} \ar[rd]_{l_b} & & (A\otimes 1)\otimes B \ar[ld]^{r_A} \\
& A\otimes B}

which is commutative by the definition of $\cN^1_l(\cG)$.
\end{proof}
\begin{rem}
\end{rem}
The constructions of the following sections can be made unital (monoidal instead of just semi-groupal). We will not be doing it for the sake of simplicity.

\subsection{Functoriality}\label{funcnuc}
Here we assign, to an invertible natural self-transformation (an automorphism) $c$ of the tensor product functor
$\otimes:\cG\times\cG\to\cG$, a strict semigroupal autoequivalence $\cN_l(c):\cN_l(\cG)\to\cN_l(\cG)$ of the nucleus. For
an object $(A,a)$ of $\cN_l(\cG)$ set $\cN_l(c)(A,a)$ to be $(A,a^c)$ where $a^c$ is defined by the diagram:

\xymatrix{A\otimes(X\otimes Y)\ar[d]^{a_{X,Y}}\ar[r]^{c_{A,X\otimes Y}} & A\otimes(X\otimes Y) \ar[r]^{A\otimes
c_{X,Y}} & A\otimes(X\otimes Y) \ar[d]^{a^c_{X,Y}}
\\
(A\otimes X)\otimes Y \ar[r]^{c_{A\otimes X,Y}} & (A\otimes X)\otimes Y \ar[r]^{c_{A,X}\otimes Y} &
(A\otimes X)\otimes Y }

\begin{prop}
The functor
$\cN_l(c)$ is strict semigroupal.
\end{prop}
\begin{proof}

The following commutative diagram shows that $(a|b)^c = (a^c|b^c)$ for $(A,a),(B,b)$ from $\cN_l(\cG)$:

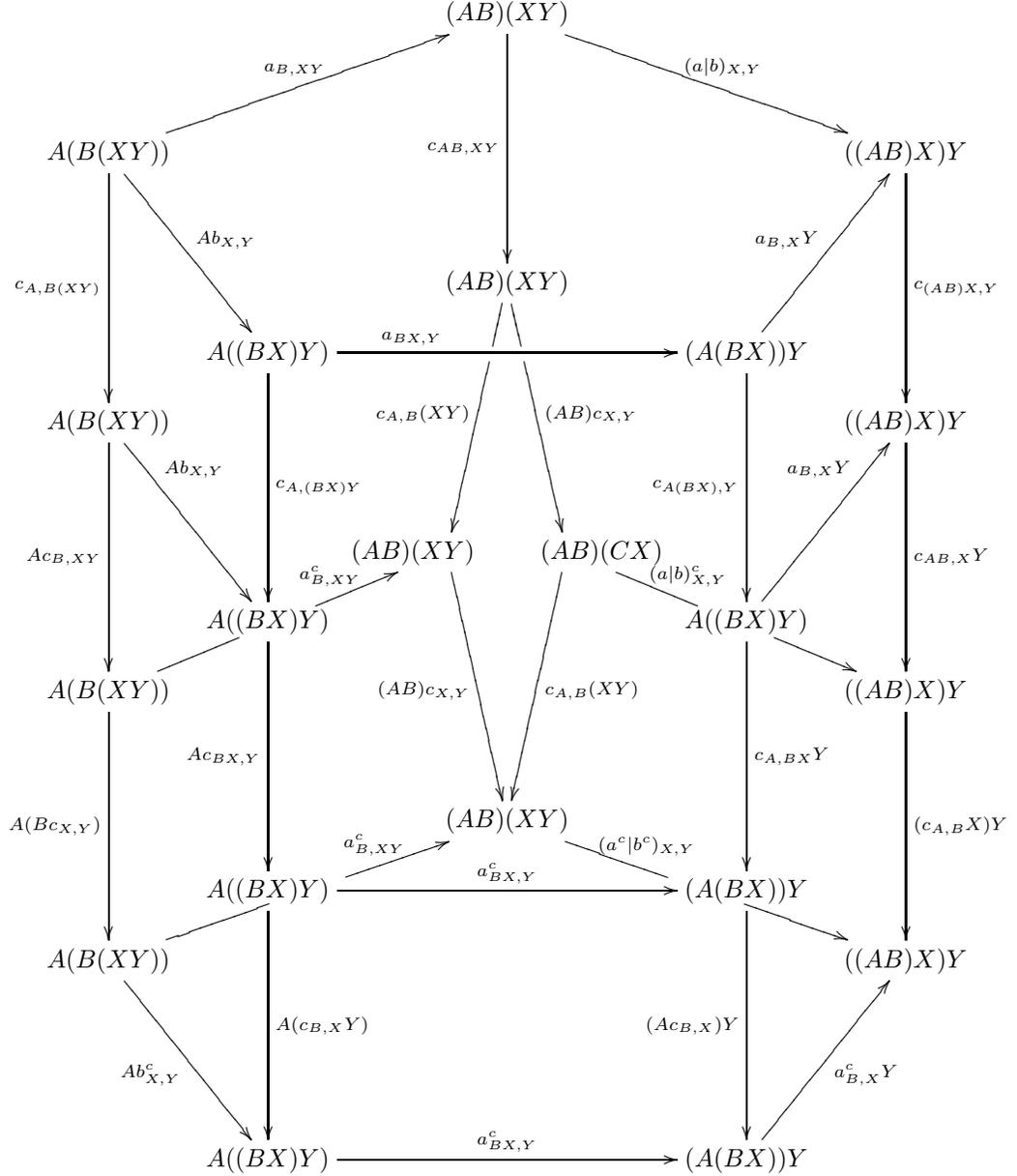
$$
\xygraph{ !{0;/r6.5pc/:;/u6.5pc/::}[]*+{A(B(XY))}="l" (
  :[u(.7)rr]*+{(AB)(XY)} ^{a_{B,XY}}
  ( :[d(.7)rr]*+{((AB)X)Y}="r" ^{(a|b)_{X,Y}}
    :[d(1.35)]*+{((AB)X)Y}="r1" ^{c_{(AB)X,Y}}
    :[d(1.35)]*+{((AB)X)Y}="r2"  ^{c_{AB,X}Y}
    :[d(1.35)]*+{((AB)X)Y}="r3" ^{(c_{A,B}X)Y}
    ,
    :[d(1.35)]*+{(AB)(XY)}="t1" _{c_{AB,XY}}
    ( :[l(.3)d(1.35)]*+!R(.5){(AB)(XY)}="tl2" _{c_{A,B}(XY)} |(.263){\hole}
      :[r(.3)d(1.35)]*+{(AB)(XY)}="t3" _{(AB)c_{X,Y}}
      : "r3" ^(.3){(a^c|b^c)_{X,Y}} |*++++{\hole}
      ,
      :[r(.3)d(1.35)]*+!L(.5){(AB)(CX)}="tr2" ^{(AB)c_{X,Y}} |(.263){\hole}
      ( : "t3" ^{c_{A,B}(XY)}
        ,
        : "r2" ^(.3){(a|b)^c_{X,Y}} |*++++{\hole}
      )
    )
  )
  ,
  :[r(.8)d]*+{A((BX)Y)}="lb" ^{Ab_{X,Y}}
  ( :[r(2.4)]*+{(A(BX))Y}="rb" ^(.3){a_{BX,Y}}
    ( : "r" ^{a_{B,X}Y}
      ,
      :[d(1.35)]*+{A((BX)Y)}="rb1" _{c_{A(BX),Y}}
      ( : "r1" ^(.7){a_{B,X}Y}
        ,
        :[d(1.35)]*+{(A(BX))Y}="rb2" ^{c_{A,BX}Y}
        :[d(1.35)]*+{(A(BX))Y}="rb3" _{(Ac_{B,X})Y}
        : "r3" _{a^c_{B,X}Y}
      )
    )
    ,
    :[d(1.35)]*+{A((BX)Y)}="lb1" ^{c_{A,(BX)Y}}
    :[d(1.35)]*+{A((BX)Y)}="lb2" _{Ac_{BX,Y}}
    ( :[d(1.35)]*+{A((BX)Y)}="lb3" ^{A(c_{B,X}Y)}
      : "rb3" ^{a^c_{BX,Y}}
      ,
      : "rb2" ^{a^c_{BX,Y}}
    )
  )
  ,
  :[d(1.35)]*+{A(B(XY))}="l1" _{c_{A,B(XY)}}
  ( : "lb1" ^(.3){Ab_{X,Y}}
    ,
    :[d(1.35)]*+{A(B(XY))}="l2" _{Ac_{B,XY}}
    ( : "tl2" ^(.7){a^c_{B,XY}} |*++++{\hole}
      ,
      :[d(1.35)]*+{A(B(XY))}="l3" _{A(Bc_{X,Y})}
      ( : "t3" ^(.7){a^c_{B,XY}} |*++++{\hole}
        ,
        : "lb3" _{Ab^c_{X,Y}}
      )
    )
  )
)
}$$ Thus $\cN_l(c)((A,a)\otimes (B,b))$ {\em coincides} with $\cN_l(c)(A,a)\otimes\cN_l(c)(B,b)$.
\end{proof}
\begin{rem}
\end{rem}
A diagram of that shaped appeared in \cite{st} in connection with $A_\infty$-maps.

By naturality, $c_{A,X\otimes Y}$ commutes with $A\otimes d_{X,Y}$ and $c_{A\otimes X,Y}$ commutes with
$d_{A,X}\otimes Y$ for any $c,d\in Aut(\otimes)$. Thus we have that $(a^c)^d = a^{cd}$ which means that the
composition $\cN_l(c)\cN_l(d)$ {\em coincides} with $\cN_l(cd)$. In other words, the assignment $c\mapsto\cN_l(c)$
defines a group homomorphism $Aut(\otimes)\to Aut^\otimes(\cN_l(\cG))$ into the group of monoidal
autoequivalences.

Now we establish the link between natural transformations of the identity functor on $\cG$ and monoidal natural
transformations of functors of the form $\cN_l(c)$. First we define an action of the group $Aut(id_\cG)$ of
automorphisms of the identity functor on the group $Aut(\otimes)$ of automorphisms of the tensor product functor.
For $c\in Aut(\otimes)$ and $f\in Aut(id_\cG)$, define $c^f\in Aut(\otimes)$ by $$c^f_{X,Y} = (f_X\otimes
f_Y)c_{X,Y}f_{X\otimes Y}^{-1}.$$ Note that by naturality of $c$, $c_{X,Y}$ commutes with $f_X\otimes f_Y$ and by
naturality of $f$, $c_{X,Y}$ commutes with $f_{X\otimes Y}$ so $c^f_{X,Y}$ coincides with $$(f_X\otimes
f_Y)f_{X\otimes Y}^{-1}c_{X,Y} = c_{X,Y}(f_X\otimes f_Y)f_{X\otimes Y}^{-1}.$$ This in particular implies that
$Aut(id_\cG)$ acts on the group $Aut(\otimes)$: $$c^{fg}_{X,Y} = c_{X,Y}(fg_X\otimes fg_Y)fg_{X\otimes Y}^{-1} =
c_{X,Y}(f_X\otimes f_Y)f_{X\otimes Y}^{-1}(g_X\otimes g_Y)g_{X\otimes Y}^{-1} = (c^f)^g_{X,Y}.$$

\begin{prop}
A natural transformation $f\in Aut(id_\cG)$ of the identity functor defines a natural transformation $\cN_l(c)\to\cN_l(c^f)$.
\end{prop}
\begin{proof}
We need to show that for
any $(A,a)\in\cN_l(\cG)$ the map $f_A:A\to A$ is a morphism $(A,a)\to (A,a^{c^f})$ in $\cN_l(\cG)$ or that the square

\xymatrix{ A\otimes(X\otimes Y) \ar[d]_{f_A\otimes(X\otimes Y)}\ar[r]^{a^c_{X,Y}} & (A\otimes X)\otimes Y
\ar[d]^{(f_A\otimes X)\otimes Y}
\\
A\otimes(X\otimes Y) \ar[r]^{a^{c^f}_{X,Y}} & (A\otimes X)\otimes Y }

is commutative. The commutativity of this square is equivalent to the commutativity of the bottom square of
the following diagram:

\xymatrix@C=-10pt{ & A\otimes(X\otimes Y) \ar'[d][dd]^{c^f_{A,X\otimes Y}}\ar[rr]^{a_{X,Y}} & & (A\otimes
X)\otimes Y \ar[dd]^{c^f_{A\otimes X,Y}}
\\
A\otimes(X\otimes Y) \ar[ur]^{f_{A\otimes(X\otimes Y)}}\ar[dd]_{c_{A,X\otimes Y}}\ar[rr]^<<<<{a_{X,Y}} & & (A\otimes
X)\otimes Y \ar[ur]^{f_{A\otimes X)\otimes Y}}\ar[dd]^<<<<{c_{A\otimes X,Y}}
\\
 & A\otimes (X\otimes Y) \ar[dd]_>>>>{A\otimes c^f_{X,Y}} & & (A\otimes X)\otimes Y \ar[dd]^{c^f_{A,X}\otimes Y}
\\
A\otimes(X\otimes Y) \ar[ur]^{f_A\otimes f_{X\otimes Y}}\ar[dd]_{A\otimes c_{X,Y}} & & (A\otimes X)\otimes Y
\ar[ur]^{f_{A,X}\otimes f_Y}\ar[dd]_<<<<{c_{A,X}\otimes Y}
\\
 & A\otimes(X\otimes Y) \ar'[r][rr]^<<<<{a^{c^f}_{X,Y}} & & (A\otimes X)\otimes Y
\\
A\otimes(X\otimes Y) \ar[rr]_{a^c_{X,Y}}\ar[ur]_{f_A\otimes(f_X\otimes f_Y)} & & (A\otimes X)\otimes Y
\ar[ur]_{(f_A\otimes f_X)\otimes f_Y} }
\end{proof}

It is clear that the natural transformation $\cN_l(f):\cN_l(c)\to\cN_l(c^f)$ is monoidal iff $f\in Aut(id_\cG)$ is
monoidal, i.e. iff $f_{X\otimes Y}=f_X\otimes f_Y$ for any $X,Y\in\cG$.

\subsection{Subcategories of nuclei}\label{subnucl}

Here we define semigroupal (monoidal) subcategories in nuclei anchored at quasi-monoidal functors. Let $\cG$ be a
magmoidal category, let $\caH$ be a monoidal (semi-groupal) category and let $F:\cG\to \caH$ be a functor equipped with
a natural collection of isomorphisms $F_{X,Y}:F(X\otimes Y)\to F(X)\otimes F(Y)$ (a {\em quasi-monoidal} functor). Define a full subcategory $\cN_l(F)$
of $\cN_l(\cG)$ consisting of pairs $(A,a)$ such that the following diagram commutes:
\begin{equation}\label{subcoh}
\xymatrix{F(A\otimes(X\otimes Y)) \ar[d]_{F(a_{X,Y})}\ar[r]^{F_{X,Y}} & F(A)\otimes F(X\otimes Y) \ar[rr]^{F(A)\otimes
F_{X,Y}} & & F(A)\otimes (F(X)\otimes F(Y)) \ar[d]_{\psi_{F(A),F(X),F(Y)}}
\\
F((A\otimes X)\otimes Y) \ar[r]^{F_{A\otimes X,Y}} & F(A\otimes X)\otimes F(X) \ar[rr]^{F_{X,A}\otimes F(Y)} & &
(F(A)\otimes F(X))\otimes F(Y) } 
\end{equation}
here $\psi$ is the associativity constraint of $\caH$.

\begin{prop}
The subcategory $\cN_l(F)$ is closed under tensor product in $\cN_l(\cG)$. The quasi-monoidal functor $F$ lifts to a
monoidal (semi-groupal) functor $\overline F:\cN_l(F)\to \caH$.
\end{prop}
\begin{proof}
We need to verify that the tensor product $(A\otimes B,a|b)$ of two pairs $(A,a),\ (B,b)$ from $\cN_l(F)$ is in
$\cN_l(F)$. It follows from the commutativity of the diagram:

$$
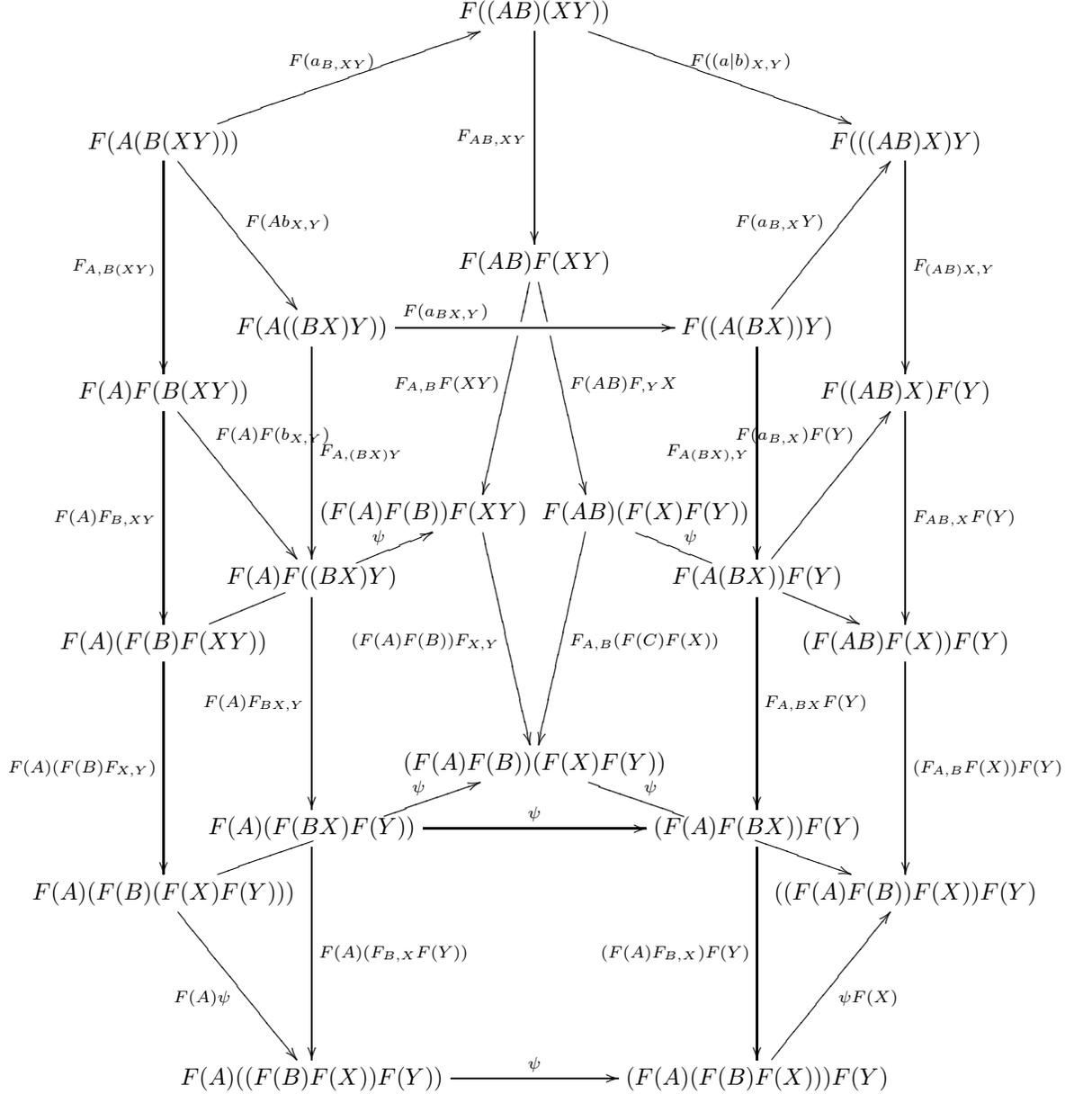
\xygraph{ !{0;/r6.5pc/:;/u6.5pc/::}[]*+{F(A(B(XY)))}="l" (
  :[u(.7)rr]*+{F((AB)(XY))} ^{F(a_{B,XY})}
  ( :[d(.7)rr]*+{F(((AB)X)Y)}="r" ^{F((a|b)_{X,Y})}
    :[d(1.35)]*+{F((AB)X)F(Y)}="r1" ^{F_{(AB)X,Y}}
    :[d(1.35)]*+{(F(AB)F(X))F(Y)}="r2"  ^{F_{AB,X}F(Y)}
    :[d(1.35)]*+{((F(A)F(B))F(X))F(Y)}="r3" ^{(F_{A,B}F(X))F(Y)}
    ,
    :[d(1.35)]*+{F(AB)F(XY)}="t1" _{F_{AB,XY}}
    ( :[l(.3)d(1.35)]*+!R(.5){(F(A)F(B))F(XY)}="tl2" _{F_{A,B}F(XY)} |(.263){\hole}
      :[r(.3)d(1.35)]*+{(F(A)F(B))(F(X)F(Y))}="t3" _{(F(A)F(B))F_{X,Y}}
      : "r3" ^(.3){\psi} |*++++{\hole}
      ,
      :[r(.3)d(1.35)]*+!L(.5){F(AB)(F(X)F(Y))}="tr2" ^{F(AB)F_{,Y}X} |(.263){\hole}
      ( : "t3" ^{F_{A,B}(F(C)F(X))}
        ,
        : "r2" ^(.3){\psi} |*++++{\hole}
      )
    )
  )
  ,
  :[r(.8)d]*+{F(A((BX)Y))}="lb" ^{F(Ab_{X,Y})}
  ( :[r(2.4)]*+{F((A(BX))Y)}="rb" ^(.3){F(a_{BX,Y})}
    ( : "r" ^{F(a_{B,X}Y)}
      ,
      :[d(1.35)]*+{F(A(BX))F(Y)}="rb1" _{F_{A(BX),Y}}
      ( : "r1" ^(.7){F(a_{B,X})F(Y)}
        ,
        :[d(1.35)]*+{(F(A)F(BX))F(Y)}="rb2" ^{F_{A,BX}F(Y)}
        :[d(1.35)]*+{(F(A)(F(B)F(X)))F(Y)}="rb3" _{(F(A)F_{B,X})F(Y)}
        : "r3" _{\psi F(X)}
      )
    )
    ,
    :[d(1.35)]*+{F(A)F((BX)Y)}="lb1" ^{F_{A,(BX)Y}}
    :[d(1.35)]*+{F(A)(F(BX)F(Y))}="lb2" _{F(A)F_{BX,Y}}
    ( :[d(1.35)]*+{F(A)((F(B)F(X))F(Y))}="lb3" ^{F(A)(F_{B,X}F(Y))}
      : "rb3" ^{\psi}
      ,
      : "rb2" ^{\psi}
    )
  )
  ,
  :[d(1.35)]*+{F(A)F(B(XY))}="l1" _{F_{A,B(XY)}}
  ( : "lb1" ^(.3){F(A)F(b_{X,Y})}
    ,
    :[d(1.35)]*+{F(A)(F(B)F(XY))}="l2" _{F(A)F_{B,XY}}
    ( : "tl2" ^(.7){\psi} |*++++{\hole}
      ,
      :[d(1.35)]*+{F(A)(F(B)(F(X)F(Y)))}="l3" _{F(A)(F(B)F_{X,Y})}
      ( : "t3" ^(.7){\psi} |*++++{\hole}
        ,
        : "lb3" _{F(A)\psi}
      )
    )
  )
)
}$$ Now define the functor $\overline F:\cN_l(F)\to \caH$ by $\overline F(A,a) = F(A)$. The coherence axiom for the
compatibility isomorphism $F_{X,Y}:F(X\otimes Y)\to F(X)\otimes F(Y)$ follows from the definition of $\cN_l(F)$.
\end{proof}

In the next proposition we explain how monoidal autoequivalences of nuclei (defined in section \ref{funcnuc}) permute the subcategories of nuclei. Let $c_{X,Y}:X\otimes Y\to X\otimes Y$ be a natural collection of isomorphisms defining a semi-groupal functor $\cN_l(c):\cN_l(\cG)\to \cN_l(\cG)$. Let $F:\cG\to\caH$ be a quasi-monoidal functor defining a semi-groupal subcategory $\cN_l(F)$. Define a new quasi-monoidal structure $F^c$ on $F$ by 

\xymatrix@C=+35pt{F^c_{X,Y}:F(X\otimes Y) \ar[r]^(.6){F(c_{X,Y})} & F(X\otimes Y) \ar[r]^(.5){F_{X,Y}} & F(X)\otimes F(Y) }
\begin{prop}
The semigroupal functor $\cN_l(c)$ defines an equivalence between the semi-groupal subcategories $\cN_l(F),\ \cN_l(F^c)$. 
\end{prop}
\begin{proof}
We need to show that the coherence diagram (\ref{subcoh}) commutes if we replace $F_{X,Y}$ by $F^c_{X,Y}$ and $a_{X,Y}$ by $a^c_{X,Y}$. This follows from the commutative diagram:
$$\xygraph{ !{0;/r5.5pc/:;/u4.5pc/::}[]*+{F(A(XY))}
( :[rr]*+{F(A)F(XY)}="mu" ^{F^c_{X,Y}}
  ( :[rr]*+{F(A)(F(X)F(Y))}="ru" ^{F(A)F^c_{X,Y}}
    :[d(1.3)]*+{(F(A)F(X))F(Y)}="rd" ^{\psi_{F(A),F(X),F(Y)}}
  ,
    :[rd]*+{F(A)F(XY)}="ur" ^{F(A)F(c_{X,Y})}
    :"ru" _{F(A)F_{X,Y}}
  )
,
  :[rd]*+{F(A(XY))} ^{F(c_{A,XY})}
  ( :"mu" _{F_{A,XY}}
  ,
    :[rd]*+{F(A(XY))} _{F(Ac_{X,Y})}
    ( :"ur" _{F_{A,XY}}
    ,
      :[d(1.3)]*+{F((AX)Y)}="d" ^{F(a_{X,Y})}
      :[ur]*{F(AX)F(Y)}="dr" _{F_{AX,Y}}
      :"rd" _{F_{A,X}F(Y)}
    )
  )
,
  :[d(1.3)]*+{F((AX)Y)} _{F(a^c_{X,Y})} 
  ( :[rr]*+{F(AX)F(Y)}="md" _{F^c_{X,Y}} |(.65){\hole}
    ( :"rd" _{F^c_{A,X}F(Y)} |(.35){\hole}
    ,
      :"dr" ^(.7){F(c_{A,X})F(Y)} |(.35){\hole}
    )
  ,
    :[rd]*+{F((AX)Y)} _{F(c_{AX,Y})}
    ( :"md" _(.3){F_{AX,Y}} |(.65){\hole}
    ,
      :"d" _{F(c_{A,X}Y)}
    )
  )
)
}$$
\end{proof}

\subsection{Actions of nuclei}\label{actnucl}

Recall that a monoidal (semigroupal) category $\cC$ {\em acts} on a category $\cM$ if there is given a monoidal
(semigroupal) functor $\cC\to \End(\cM)$ into the category of endofunctors on $\cM$. In that case $\cM$ is called a (left) {\em module} category over $\cC$ or a (left) $\cC$-{\em category}.

Let $L:\cC\to \End(\cM)$, $L':\cC\to \End(\cM')$ be two actions. 
A functor $F:\cM\to \cM'$ between $\cC$-categories is a $\cC$-{\em functor} if it is equipped with a collection of isomorphisms $F_{A,X}:F(L(A)(X))\to L'(A)(F(X))$ natural in $A\in\cC, X\in\cM$ such that the diagram 
\begin{equation}\label{cohfun}
\xymatrix@C=+35pt{ F(L(AB)(X)) \ar[r]^{F_{AB,X}} \ar[d]^{F(L_{A,B}(X))} & L'(AB)(F(X)) \ar[r]^{L'_{A,B}(F(X))} & L'(A)(L'(B)(F(X))) \ar@{=}[d] \\ 
F(L(A)(L(B)(X))) \ar[r]^{F_{A,L(B)(X)}} & L'(A)(F(L(B)(X))) \ar[r]^{L'(A)(F_{B,X})} & L'(A)(L'(B)(F(X))) }
\end{equation}
commutes. 

For a magmoidal category $\cG$ define a functor $L:\cN_l(\cG)\to \End(\cG)$ by $L(A,a)(X) = A\otimes X$. Define a monoidal structure $L_{(A,a),(B,b)}:L((A,a)\otimes(B,b))\to L(A,a)\circ L(B,b)$ (a natural collection of isomorphisms of functors) by $$\xymatrix{ L((A,a)\otimes(B,b))(X) \ar[rr]^{L_{(A,a),(B,b)}(X)} \ar@{=}[d] & & L(A,a)(L(B,b)(X)) \ar@{=}[d] \\ (A\otimes B)\otimes X & & \ar[ll]_{a_{B,X}} A\otimes(B\otimes X) }$$
\begin{prop}
The functor $L:\cN_l(\cG)\to \End(\cG)$ is semigroupal.
\end{prop}
\begin{proof}
All we need to check is the coherence axiom for the monoidal structure, which follows from the definition of the tensor product in $\cN_l(\cG)$: 

\xymatrix@C=-5pt{ L((A,a)((B,b)(C,c)))(X) \ar[rr]^{L(\phi_{(A,a),(B,b),(C,c)})(X)} \ar[dd]^>>>>>>>>{L_{(A,a),(B,b)(C,c)}(X)} \ar@{=}[dr] & &  L(((A,a)(B,b))(C,c))(X) \ar[dd]^>>>>>>>>{L_{(A,a)(B,b),(C,c)}(X)} \ar@{=}[dr] & \\ 
& (A(BC))X \ar[rr]^<<<<<<<<<<{a_{B,C}X}  |(.5){\hole} & & ((AB)C)X  \\ 
L(A,a)(L((B,b)(C,c))(X)) \ar[dd]^>>>>>>>>{L(A,a)L_{(B,b),(C,c)}(X)} \ar@{=}[dr] &  & L((A,a)(B,b))(L(C)(X)) \ar[dd]^>>>>>>>>{L_{(A,a),(B.b)}(L_{(C,c)}(X))} \ar@{=}[dr] & \\ 
& A((BC)X) \ar[uu]_{a_{BC,X}} & & (AB)(CX) \ar[uu]_{(a|b)_{C,X}} \\
L(A,a)(L(B,b)(L(C,c)(X))) \ar@{=}[rr] \ar@{=}[rd] & & (L(A,a)L(B,b))(L(C,c)(X)) \ar@{=}[rd] & \\ 
& A(B(CX)) \ar[uu]_<<<<<<<<{Ab_{C,X}}  |(.5){\hole} \ar@{=}[rr] & & A(B(CX)) \ar[uu]_{a_{B,CX}} } 
\end{proof}

Now consider the cartesian square $\cG\times \cG$ of a magmoidal category as a left $\cN_l(\cG)$-category by making the semigroupal category $\cN_l(\cG)$ to act on the first factor of $\cG\times \cG$. Define a natural collection of isomorphism $\otimes_{(A,a),X,Y}:L(A,a)(X)\otimes Y\to L(A,a)(X\otimes Y)$ by 
\begin{equation}\label{actfun}
\xymatrix@C=35pt{ L(A,a)(X)\otimes Y \ar[r]^{\otimes_{(A,a),X,Y}} \ar@{=}[d] & L(A,a)(X\otimes Y) \ar@{=}[d] \\
(A\otimes X)\otimes Y & A\otimes (X\otimes Y) \ar[l]_{a_{X,Y}} }
\end{equation}
\begin{prop}
The functor $\otimes:\cG\times\cG\to \cG$ between $\cN_l(\cG)$-categories is a $\cN_l(\cG)$-functor.
\end{prop}
\begin{proof}
We need to check that the natural collection (\ref{actfun}) satisfies the coherence condition (\ref{cohfun}). It follows from the diagram: 

\xymatrix@C=-4pt{ L((A,a)(B,b))(XY) \ar[rr]^{L_{(A,a),(B,b)}(X)Y} \ar[dd]^>>>>>>>>{\otimes_{(A,a)(B,b),X,Y}} \ar@{=}[dr] & &  L(A,a)(L(B,b)(X))Y \ar[dd]^>>>>>>>>{\otimes_{(A,a),L(B,b)(X),Y}} \ar@{=}[dr] & \\ 
& ((AB)X)Y \ar[rr]^<<<<<<<<<<{a_{B,X}Y}  |(.5){\hole} & & (A(BX))Y  \\ 
L((A,a)(B,b))(XY)\ar[dd]^>>>>>>>>{L_{(A,a),(B.b)}(XY)}\ar@{=}[dr] &  & L(A,a)(L(B,b)(X)Y) \ar[dd]^>>>>>>>>{L(A,a)(\otimes_{(B,b),X,Y})}  \ar@{=}[dr] & \\ 
& (AB)(XY) \ar[uu]_<<<<<<<<{(a|b)_{X,Y}} & & \A((BX)Y) \ar[uu]_{a_{BX,Y}} \\
L(A,a)(L(B,b)(XY)) \ar@{=}[rr] \ar@{=}[rd] & & (L(A,a)L(B,b))(XY)) \ar@{=}[rd] & \\ 
& A(B(XY)) \ar[uu]_<<<<<<<<{a_{B,XY}}  |(.5){\hole} \ar@{=}[rr] & & A(B(XY))  \ar[uu]_{Ab_{X,Y}} } 
\end{proof}

In the next proposition we characterise left nucleus $\cN_l(\cG)$ as a category of certain endofunctors on $\cG$. For a magmoidal category $\cG$ a functor $F:\cG\to \cG$ will be called (right) $\cG$-{\em invariant} if it is equipped with a natural collection of isomorphisms $f_{X,Y}:F(X\otimes Y)\to F(X)\otimes Y$. A {\em morphism} $c:F\to G$ of $\cG$-invariant functors is a natural transformation such that the diagram
\xymatrix{ F(X\otimes Y) \ar[r]^{f_{X,Y}} \ar[d]^{c_{X\otimes Y}} & F(X)\otimes Y \ar[d]^{c_X\otimes Y} \\
G(X\otimes Y) \ar[r]^{g_{X,Y}}  & G(X)\otimes Y } 

commutes. Define $\End_\cG(\cG)$ to be the category of $\cG$-invariant functors. 
The composite of two $\cG$-invariant functors $(F,f),(G,g)$ has a structure of $\cG$-invariant functor: 

\xymatrix{ (F\circ G)(X\otimes Y) \ar[rr]^{(f\circ g)_{X,Y}} \ar@{=}[d] & & (F\circ G)(X)\otimes Y \ar@{=}[d] \\
F(G(X\otimes Y)) \ar[r]^{F(g_{X,Y})} & F(G(X)\otimes Y) \ar[r]^{f_{G(X),Y}} & F(G(X))\otimes Y }
\begin{lem}
The category $\End_\cG(\cG)$ is strict monoidal. 
\end{lem}
\begin{proof}
All we need to verify is that the horizontal composition of morphisms of $\cG$-invariant functors preserves $\cG$-invariant structure of the composition. This follows from the commutative diagram: 
$$\xygraph{ !{0;/r8.5pc/:;/u6.5pc/::}[]*+{(F\circ G)(XY)}
( :@{=}[r(.5)u(.5)]*+{F(G(XY))} 
  ( :[r(1.2)]*+{F(G(X)Y)} ^{F(g_{X,Y})}
    ( :[r(1.2)]*+{F(G(X))Y}="ru" ^{f_{G(X),Y}}
      :[d]*+{F'(G(X))Y}="rm" ^{c_{G(X)}Y}
      :[d]*+{F'(G'(X))Y}="rd" ^{F(d_X)Y}
    ,
      :[r(.3)d]*+{F'(G(X)Y)} ^<<<<<<<<{c_{G(X)Y}}  |(.5){\hole}
      ( : "rm" ^{f'_{G(X),Y}}  |(.445){\hole}
      ,
        :[l(.3)d]*+{F'(G'(X)Y)}="md" ^{F'(d_XY)}
        : "rd" ^{f'_{G'(X),Y}}  |(.58){\hole}
      )
    ,
      :[l(.3)d]*+{F(G'(X)Y)}="m" _<<<<<<<<{F(d_XY)} |(.5){\hole}
      : "md" _{c_{G'(X)Y}}
    )
  ,
    :[d]*+{F(G'(XY))} ^<<<<<<<{F(d_{XY})} |(.5){\hole}
    ( : "m" ^{F(g'_{X,Y")}}
    ,
     :[d]*+{F'(G'(XY))}="l" ^{c_{G'(XY)}}
     : "md" ^{F'(g'_{X,Y})}
    )
  )
,
  :[r(2.4)]*+{(F\circ G)(X)Y} ^{(f\circ g)_{X,Y}}
  ( :@{=}"ru" 
  ,
    :[dd]*+{(F'\circ G')(X)Y}="dr" 
    :@{=}"rd" 
  )
,
  :[dd]*+{(F'\circ G')(XY)} ^{(c\circ d)_{XY}} 
  ( :@{=}"l" 
  ,
    :"dr" ^{(f'\circ g')_{X,Y}}
  )
)
}$$
\end{proof}

Now we extend the (left) action $L:\cN_l(\cG)\to \End(\cG)$ to a functor $L:\cN_l(\cG)\to \End_\cG(\cG)$ by the assignment $(A,a)\mapsto (L(A),L(a))$, where the structure $L(a)$ of $\cG$-invariant functor on $L(A)$ is defined as follows: 

\xymatrix{ L(A)(X\otimes Y) \ar[r]^{L(a)_{X,Y}} \ar@{=}[d] & L(A)(X)\otimes Y  \ar@{=}[d] \\
A\otimes(X\otimes Y) \ar[r]^{a_{X,Y}} & (A\otimes X)\otimes Y } 

\begin{prop}
Let $\cG$ be a unital magmoidal category.
Then the functor $L:\cN^1_l(\cG)\to \End_\cG(\cG)$ is a monoidal equivalence. 
\end{prop}
\begin{proof}
First we note that the monoidal structure $L_{X,Y}$ of the functor $L:\cN^1_l(\cG)\to \End(\cG)$ is compatible with the $\cG$-invariant structures (a morphism in $\End_\cG(\cG)$) thus making the extended functor $L:\cN^1_l(\cG)\to \End_\cG(\cG)$ monoidal. This follows from the commutative diagram:

$$
\xygraph{ !{0;/r6pc/:;/u6pc/::}[]*+{L(AB)(XY)}="tl" 
(
  :[u(.7)rr]*+{L(AB)(X)Y} ^{L(a|b)_{X,Y}}
  (
    :[d(.7)rr]*+{L(A)(L(B)(X))Y}="ur" ^{L_{A,B}(X)Y} 
    :@{=}[d(.85)]*+{A((BX)Y)}="dr" 
    :[u(.7)ll]*+{((AB)X)Y}="du" ^{a_{B,X}Y}
  ,
    :@{=}"du"
  )
,
  :@{=}[d(.85)]*+{(AB)(XY)}="dl" 
  :"du" _{(a|b)_{X,Y}}
,
  :[r(.8)d]*+{L(A)(L(B)(XY))}="udl" ^(.4){L_{A,B}(XY)}
  (
    :[r(2.4)]*+{L(A)(L(B)(X)Y)}="udr" ^{L(A)((L(b)_{X,Y})}
    :"ur" ^(.6){L(a)_{L(B)(X),Y}}
  ,
    :@{=}[d(.85)]*+{A(B(XY))}="ddl"
    ( :"dl" ^{a_{B,XY}}
    ,
      :[r(2.4)]*+{A((BX)Y)} ^{Ab_{X,Y}}
      ( :"dr" _{a_{BX,Y}}
      ,
        :@{=}"udr"
      )
    )
  )
)
}$$

To see that $L$ is an equivalence it is enough to note that the assignment $(F,f)\mapsto (F(1),\overline f)$ defines a functor quisi-inverse to $L$. Here $\overline f_{X,Y}:F(1)\otimes (X\otimes Y) \to (F(1)\otimes X)\otimes Y$ is defined by 

\xymatrix@C=40pt{ F(1)\otimes(X\otimes Y) \ar[d]^{\overline f_{X,Y}} & F(1\otimes(X\otimes Y)) \ar[l]_{F_{1,X\otimes Y}} \ar[r]^{F(l_{X\otimes Y})} & F(X\otimes Y) \ar[d]^{F_{X,Y}} \\
(F(1)\otimes X)\otimes Y & F(1\otimes X)\otimes Y \ar[l]_{F_{1,X}\otimes Y} \ar[r]^{F(l_X)\otimes Y} & F(X)\otimes Y }
\end{proof}

\section{Multiplicants of functors}

Now let $F:\cG\to\caH$ be a functor between semigroupal (monoidal) categories. Define $\cM_l(F)$ to be the category of pairs $(A,a)$,
where $A\in\cG$ and $a$ is a natural family of isomorphisms $a_X:F(A\otimes X)\to F(A)\otimes F(X)$. A morphism from
$(A,a)$ to $(B,b)$ is a morphism $f:A\to B$ in $\cG$ such that the following diagram commutes for each $X\in\cG$
\xymatrix{ F(A\otimes X) \ar[d]_{F(f\otimes X)}\ar[r]^{a_X} & F(A)\otimes F(X) \ar[d]^{F(f)\otimes F(X)}
\\
F(B\otimes X) \ar[r]^{b_X} & F(A)\otimes F(X) }

Note that the assignment $(A,a)\mapsto A$ defines a functor $\cM_l(F)\to \cG$ (the {\em forgetful} functor).

We can define a tensor product on $\cM_l(F)$ by $(A,a)\otimes (B,b) = (A\otimes B,a|b)$, where the natural family
$a|b$ is given by the diagram

\xymatrix{F(A\otimes(B\otimes X)) \ar[d]_{F(\phi_{A,B,X})}\ar[r]^{a_{B\otimes X}} & F(A)\otimes F(B\otimes X)
\ar[rr]^{F(A)\otimes b_X} & & F(A)\otimes (F(B)\otimes F(X)) \ar[d]_{\psi_{F(A),F(B),F(X)}}
\\
F((A\otimes B)\otimes X) \ar[r]^{(a|b)_X} & F(A\otimes B)\otimes F(X) \ar[rr]^{a_B\otimes F(X)} & & (F(A)\otimes
F(B))\otimes F(X) }

\begin{theo}
The category $\cM_l(F)$ is semigroupal. The forgetful functor $\cM_l(F)\to \cG$ is a strict semigroupal functor. The composition of the forgetful functor with $F$ is a semigroupal  functor $\overline F:\cM_l(F)\to \caH$ with semigroupal structure: $$\overline F_{(A,a),(B,b)} = a_B:F(A\otimes B)\to F(A)\otimes F(B).$$
\end{theo}
\begin{proof}
The following diagram proves that the tensor product $f\otimes g$ of morphisms $f:(A,a)\to(C,c),\ g:(B,b)\to(D,d)$
is a morphism from $(A,a)\otimes(B,b)$ to $(C,c)\otimes(D,d)$:

\xymatrix@C=-20pt{ & F(C\otimes(D\otimes X)) \ar'[d][dd]^{c_{D\otimes X}}\ar[rr]^{F(\phi_{C,D,X})} & & F((C\otimes
D)\otimes X) \ar[dd]_{(c|d)_X}
\\
F(A\otimes(B\otimes X)) \ar[ur]^{F(f\otimes (g\otimes I))}\ar[dd]_{a_{B\otimes X}}\ar[rr]^{F(\phi_{A,B,X})} & &
F((A\otimes B)\otimes X) \ar[ur]^{F((f\otimes g)\otimes I)}\ar[dd]_{(a|b)_X}
\\
 & F(C)\otimes F(D\otimes X) \ar[dd]_{I\otimes d_X} & & F(C\otimes D)\otimes F(X) \ar[dd]_{c_D\otimes I}
\\
F(A)\otimes F(B\otimes X) \ar[ur]^{F(f)\otimes F(g\otimes I)}\ar[dd]_{I\otimes b_X} & & F(A\otimes B)\otimes F(X)
\ar[ur]^{F(f\otimes g)\otimes I}\ar[dd]_<<<<{a_B\otimes I}
\\
 & F(C)\otimes(F(D)\otimes F(X)) \ar'[r][rr]^<<<<{\phi_{F(C),F(D),F(X)}} & & (F(C)\otimes F(D))\otimes F(X)
\\
F(A)\otimes(F(B)\otimes F(X)) \ar[rr]_{\phi_{F(A),F(B),F(X)}}\ar[ur]^{F(f)\otimes(F(g)\otimes I)} & & (F(A)\otimes
F(B))\otimes F(X) \ar[ur]^{(F(f)\otimes F(g))\otimes I} }

It follows from the diagram

$$
\xygraph{ !{0;/r6.5pc/:;/u6.5pc/::}[]*+{F(A(B(CX)))}="l" (
  :[u(.7)rr]*+{F((AB)(CX))} ^{F(\phi)}
  ( :[d(.7)rr]*+{F(((AB)C)X)}="r" ^{F(\phi)}
    :[d(1.35)]*+{F((AB)C)F(X)}="r1" ^{((a|b)|c)_X}
    :[d(1.35)]*+{(F(AB)F(C))F(X)}="r2"  ^{(a|b)_CF(X)}
    :[d(1.35)]*+{((F(A)F(B))F(C))F(X)}="r3" ^{(a_BF(C))F(X)}
    ,
    :[d(1.35)]*+{F(AB)F(CX)}="t1" _{(a|b)_{CX}}
    ( :[l(.3)d(1.35)]*+!R(.5){(F(A)F(B))F(CX)}="tl2" _{a_BF(CX)} |(.263){\hole}
      :[r(.3)d(1.35)]*+{(F(A)F(B))(F(C)F(X))}="t3" _{(F(A)F(B))c_X}
      : "r3" ^(.3){\psi} |*++++{\hole}
      ,
      :[r(.3)d(1.35)]*+!L(.5){F(AB)(F(C)F(X))}="tr2" ^{F(AB)c_X} |(.263){\hole}
      ( : "t3" ^{a_B(F(C)F(X))}
        ,
        : "r2" ^(.3){\psi} |*++++{\hole}
      )
    )
  )
  ,
  :[r(.8)d]*+{F(A((BC)X))}="lb" ^{F(A\phi)}
  ( :[r(2.4)]*+{F((A(BC))X)}="rb" ^(.3){F(\phi)}
    ( : "r" ^{F(\phi X)}
      ,
      :[d(1.35)]*+{F(A(BC))F(X)}="rb1" _{(a|(b|c))_X}
      ( : "r1" ^(.7){F(\phi)F(X)}
        ,
        :[d(1.35)]*+{(F(A)F(BC))F(X)}="rb2" ^{a_{BC}F(X)}
        :[d(1.35)]*+{(F(A)(F(B)F(C)))F(X)}="rb3" _{(F(A)b_C)F(X)}
        : "r3" _{\psi F(X)}
      )
    )
    ,
    :[d(1.35)]*+{F(A)F((BC)X)}="lb1" ^{a_{(BC)X}}
    :[d(1.35)]*+{F(A)(F(BC)F(X))}="lb2" _{F(A)(b|c)_X}
    ( :[d(1.35)]*+{F(A)((F(B)F(C))F(X))}="lb3" ^{F(A)(b_CF(X))}
      : "rb3" ^{\psi}
      ,
      : "rb2" ^{\psi}
    )
  )
  ,
  :[d(1.35)]*+{F(A)F(B(CX))}="l1" _{a_{B(CX)}}
  ( : "lb1" ^(.3){F(A)F(\phi)}
    ,
    :[d(1.35)]*+{F(A)(F(B)F(CX))}="l2" _{F(A)b_{CX}}
    ( : "tl2" ^(.7){\psi} |*++++{\hole}
      ,
      :[d(1.35)]*+{F(A)(F(B)(F(C)F(X)))}="l3" _{F(A)(F(B)c_X)}
      ( : "t3" ^(.7){\psi} |*++++{\hole}
        ,
        : "lb3" _{F(A)\psi}
      )
    )
  )
)
}$$
that the square \xymatrix{ F((A(BC))X) \ar[rr]^{(a|(b|c))_X} \ar[d]_{F(\phi X)} & & F(A(BC))F(X) \ar[d]^{F(\phi)F(X)} \\ F(((AB)C)X) \ar[rr]^{((a|b)|c)_X} & & F((AB)C)F(X) }

commutes, which means that the associativity constraint of $\cG$
\newline
$\phi_{A,B,C}:A\otimes (B\otimes C)\to (A\otimes B)\otimes C$ is a morphims of $\cM_l(F)$: $$\phi_{(A,a),(B,b),(C,c)} = \phi_{A,B,C}.$$ Thus $\cM_l(F)$ is semigroupal and the forgetful functor $\cM_l(F)\to \cG$ is strict semigroupal. It follows from the definition of $\cM_l(F)$ that $\overline F$ is semigroupal.
\end{proof}

\begin{rem}
\end{rem}
The semigroupal category $\cM_l(F)$ becomes monoidal if we assume that the functor $F$ {\em preserves} monoidal unit, i.e. there is given an isomorphism $\epsilon:F(1)\to 1$. It can be verified that the object $(1,\iota)\in \cM_l(F)$ is a monoidal unit. Here $\iota_X$ is the composition: 

\xymatrix{ F(1\otimes X) \ar[r]^{F(l_X)} & F(X) & 1\otimes F(X)  \ar[l]_{l_{F(X)}} & F(1)\otimes F(X) \ar[l]_{\epsilon\otimes I} }

The semigroupal functors $\cM_l(F)\to \cG,\caH$ allow us to define $\cM_l(F)$-actions on the categories $\cG,\ \caH$. In the next proposition we show that the functor $F$ preserves these actions.
\begin{prop}
The functor $F:\cG\to \caH$ has a natural structure of a $\cM_l(F)$-functor.
\end{prop}
\begin{proof}
The  $\cM_l(F)$-action on $\cG$ is given by $(X,x)(Y) = X\otimes Y$ for $Y\in\cG$, while the $\cM_l(F)$-action on $\caH$ is simply $(X,x)(Z) = F(X)\otimes Z$ for $Z\in\caH$. The composition 

\xymatrix{F((X,x)(Y)) \ar@{=}[r] & F(X\otimes Y) \ar[r]^{x_Y} & F(X)F(Y) \ar@{=}[r] & (X,x)F(Y)}
defines the structure of a $\cM_l(F)$-functor on $F$. The coherence for this structure follows from the definition of tensor product in $\cM_l(F)$. 
\end{proof}

\subsection{Functoriality}
Here we define a correspondence between isomorphisms $c:F\to G$ of functors $F,G:\cG\to\caH$ between monoidal
categories and monoidal functors $\cM_l(c):\cM_l(F)\to\cM_l(G)$ between their multiplicators. For an object $(A,a)$
of $\cM_l(F)$ define $\cM_l(c)(A,a)$ to be $(A,a^c)$ where $a^c$ is defined by the diagram:

\xymatrix{F(A\otimes
X) \ar[r]^{a_X}\ar[d]_{c_{A\otimes X}} & F(A)\otimes F(X) \ar[d]^{c_A\otimes c_X}
\\
G(A\otimes X) \ar[r]^{a^c_X} & G(A)\otimes G(X) }

\begin{prop}
The functor $\cM_l(c):\cM_l(F)\to\cM_l(G)$ defined above is strict monoidal and fits into commutative diagram of monoidal functors:
\begin{equation}\label{dd}
\xymatrix{\cM_l(F) \ar[r]^{\cM_l(c)}\ar[d]\ar[rd] & \cM_l(G)\ar[d]\ar[ld]
\\
\cG & \caH }
\end{equation}
\end{prop}
\begin{proof}

The following commutative diagram shows that $(a|b)^c = (a^c|b^c)$ for $(A,a),(B,b)\in \cM_l(F)$:

\xymatrix@C=-20pt{ & G(A\otimes(B\otimes X)) \ar'[d][dd]^{a^c_{B\otimes X}}\ar[rr]^{G(\phi_{A,B,X})} & &
G((A\otimes B)\otimes X) \ar[dd]^{(a|b)^c}_{(a^c|b^c)_X}
\\
F(A\otimes(B\otimes X)) \ar[ur]^{c_{A\otimes(B\otimes X)}}\ar[dd]_{a_{B\otimes X}}\ar[rr]^{F(\phi_{A,B,X})}
& & F((A\otimes B)\otimes X) \ar[ur]^{c_{(A\otimes B)\otimes X}}\ar[dd]_{(a|b)_X}
\\
 & G(A)\otimes G(B\otimes X) \ar[dd]_{I\otimes b^c_X} & & G(A\otimes B)\otimes G(X) \ar[dd]^{a^c_B\otimes I}
\\
F(A)\otimes F(B\otimes X) \ar[ur]^{c_A\otimes c_{B\otimes X}}\ar[dd]_{I\otimes b_X} & & F(A\otimes B)\otimes F(X)
\ar[ur]^{c_{A\otimes B}\otimes c_X}\ar[dd]_<<<<{a_B\otimes I}
\\
 & G(A)\otimes(G(B)\otimes G(X)) \ar'[r][rr]^<<<<{\psi_{G(A),G(B),G(X)}} & & (G(A)\otimes G(B))\otimes G(X)
\\
F(A)\otimes(F(B)\otimes F(X)) \ar[rr]_{\psi_{F(A),F(B),F(X)}}\ar[ur]^{c_A\otimes(c_B\otimes c_X)} & & (F(A)\otimes
F(B))\otimes F(X) \ar[ur]^{(c_A\otimes c_B)\otimes c_X} }

This proves that the functor $\cM_l(c)$ is strict monoidal: $$\cM_l(c)((A,a)\otimes (B,b)) = \cM_l(c)(A\otimes B,a|b) =
(A\otimes B,(a|b)^c)$$ coincides with $$\cM_l(c)(A,a)\otimes\cM_l(c)(B,b) = (A,a^c)\otimes (B,b^c) = (A\otimes
B,a^c|b^c).$$ Commutativity of the diagram (\ref{dd}) follows from the definition of $\cM_l(c)$.
\end{proof}

Note that for successive isomorphisms of functors $F\stackrel{c}{\to} G\stackrel{d}{\to} H$ the composition of monoidal functors
$\cM_l(c)$ and $\cM_l(d)$ {\em coincides} with $\cM_l(cd)$. It follows from the fact that $a^{cd} = (a^c)^d$ for
$(A,a)\in\cM_l(F)$.

\subsection{Composition properties}
For a successive (set-theoretic) maps of monoids
\xymatrix{ A\ar[r]^f & B \ar[r]^g & C}
(left) multiplicants form a successive pair of spans (in the category of monoids):

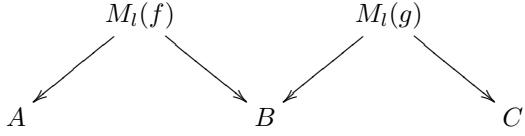
\xymatrix{ & M_l(f) \ar[ld] \ar[rd] & & M_l(g) \ar[ld] \ar[rd] & \\ A & & B & & C}

The composition of spans is given by the pullback (fibered product) of the two middle maps:
$$M_l(f)\times_B M_l(g) = \{ (a,f(a)),\ a\in M_l(f),\ f(a)\in M_l(g) \}.$$ Note that the ``projection" $(a,f(a))\mapsto a$ defines a homomorphism $M_l(f)\times_B M_l(g)\to M_l(gf)$. Indeed, for such $a$ and any $x\in A$ we have that $$gf(ax) = g(f(a)f(x)) = gf(a)gf(x).$$ Moreover, this homomorphism fits into a commutative diagram (morphism of spans):

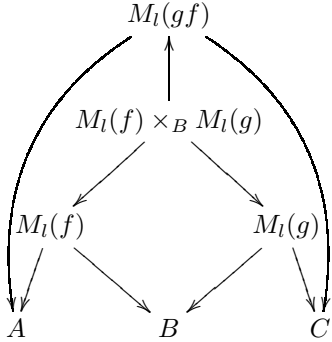
\xymatrix@C=-10pt{ & & M_l(gf) \ar@/_25pt/[llddd] \ar@/^25pt/[rrddd]  & & \\ & & M_l(f)\times_B M_l(g) \ar[ld] \ar[rd] \ar[u] & & \\ & M_l(f) \ar[ld] \ar[rd] & & M_l(g) \ar[ld] \ar[rd] & \\ A & & B & & C}

Thus multiplicants define an oplax functor from the category of monoids and set-theoretic maps into the bicategory of spans of monoids.

Here we categorify this construction by defining a monoidal functor $\cM_l(F)\times_{\caH}\cM_l(G)\to\cM_l(GF)$ for a
pair of composable functors $\cG\stackrel{F}{\to}\caH\stackrel{G}{\to}\cJ$ between monoidal categories. Here by
$\cM_l(F)\times_{\caH}\cM_l(G)$ we mean the pseudo-pullback in the 2-category of monoidal categories: with objects
$(A,a,B,b,x)$, where $(A,a)\in \cM_l(F),\ (B,b)\in \cM_l(G)$, and $x:F(A)\to B$ being an isomorphism in $\caH$. A morphism $(A,a,B,b,x)\to (A',a',B',b',x')$ is a pair $(f,g)$ of morphisms $f:(A,a)\to (A',a'),\ g:(B,b)\to (B',b')$ in $\cM_l(F),\ \cM_l(G)$ respectively, such that the square: 
\xymatrix{ F(A) \ar[r]^x \ar[d]_{F(f)} & B \ar[d]^g \\ F(A') \ar[r]^{x'} & B' }

commutes. The
tensor product in $\cM_l(F)\times_{\caH}\cM_l(G)$ has the form: $$(A,a,B,b,x)\otimes (A',a',B',b',x') = (A\otimes
A',a|a',B\otimes B',b|b',x|x'),$$ where $x|x'$ is the composition

\xymatrix{ F(A\otimes A') \ar[r]^{a_{A'}} & F(A)\otimes F(A') \ar[r]^>>>>>{x\otimes x'} & B\otimes B'.}

Define a functor $\cM_l(F,G):\cM_l(F)\times_{\caH}\cM_l(G)\to\cM_l(GF)$ by $(A,a,B,b,x)\mapsto (A,\tilde a)$, where
$\tilde a$ is the composition:

\xymatrix{ GF(A\otimes X) \ar[r]^<<<<{G(a_X)} & G(F(A)\otimes F(X)) \ar[rr]^{G(x\otimes F(X))} & & G(B\otimes F(X)) }
\xymatrix{ \ar[r]^>>>>{b_{F(X)}} & G(B)\otimes GF(X) \ar[rr]^{G(x)^{-1}\otimes GF(X)} & & GF(A)\otimes GF(X)}

\begin{prop}
The functor $\cM_l(F,G)$ is strict monoidal. 
\end{prop}
\begin{proof}
We need to show that $$\cM_l(F,G)((A,a,B,b,x)\otimes(A',a',B',b',x')) = $$ $$= \cM_l(F,G)(A\otimes
A',a|a',B\otimes B',b|b',x|x') = (A\otimes A',\widetilde{a|a'})$$ coincides with $$\cM_l(F,G)(A,a,B,b,x) \otimes\cM_l(F,G)(A',a',B',b',x') = $$ $$= (A,\tilde a)\otimes(A',\tilde a') = (A\otimes A',\tilde a|\tilde a').$$
It follows from the diagram: 

$$
\xygraph{ !{0;/r4.5pc/:;/u6.5pc/::}[]*+{GF(A(A'X))} (
  :[u(.7)r(.6)]*+{GF((AA')X)} ^{GF(\phi)}
  ( :[d(1)l(.3)]*+{G(F(AA')F(X))} ^<<<<<{G((a|a')_X)} |(.7)*+{\hole}
    ( :[d(3)r(.3)]*+{G((BB')F(X))}="ld" ^<<<<<<<<<<<<<<<{G((x|x')F(X))} |(.32){\hole} |(.407){\hole} |(.518){\hole} |(.8){\hole}
      :[rr]*+{G(BB')GF(X)}="md" ^{(b|b')_{F(X)}} |(.66){\hole}
      :[rr]*+{(G(B)G(B'))GF(X)}="rd" ^{b_{B'}GF(X)}
    ,
      :[d(1)l(.3)]*+{G((F(A)F(A'))F(X))}="lm" _{G(a_{A'}F(X))} |(.7)*+{\hole}
      :"ld" _{G((xx')F(X))} |(.12){\hole} |(.34){\hole}  |(.41){\hole} |(.727){\hole}
    )
  ,
    :[rr]*+{GF(AA')GF(X)} ^{(\tilde a|\tilde a')_X} _{(\widetilde{a|a'})_X}
    ( :"md" _<<<<<<<<<{G(x|x')GF(X)}  |(.177){\hole}  |(.263){\hole}  |(.3){\hole} |(.43){\hole} |(.508){\hole} |(.566){\hole} |(.615){\hole} |(.7)*+{\hole} |(.94){\hole}
    ,
      :[rr]*+{(GF(A)GF(A'))GF(X)}="ru" ^{\tilde a_{A'}GF(X)}
      :[d(2.5)]*+{(G(B)GF(A'))GF(X)}="rm" ^{(G(x)GF(A'))GF(X)}
      :"rd" ^<<<<<<<<<{(G(B)G(x'))GF(X)}
    ,
      :[d(1.05)r(.2)]*+{G(F(A)F(A'))GF(X)} ^<<<<<<{G(a_{A'})GF(X)}  |(.66){\hole}
      :[d(1)]*+{G(BF(A'))GF(X)} ^{G(xF(A'))GF(X)} |(.19){\hole} |(.69)*+{\hole}
      ( :"rm" ^{b_{F(A')}GF(X)} |(.7){\hole}
      ,
        :"md" ^<<<<<<<<<<<<<{G(Bx')GF(X)}  |(.106){\hole}  |(.21){\hole} |(.38)*+{\hole}  |(.88){\hole}
      )
    )
  )
,
  :[rr]*+{GF(A)GF(A'X)} ^{\tilde a_{A'X}}
  ( :[rr]*+{GF(A)(GF(A')GF(X))} ^{GF(A)\tilde a'_X}
    ( :"ru" _\psi
    ,
      :[d(2.5)]*+{GF(A)(G(B')GF(X))}="mr" ^{GF(A)(G(x')GF(X))}
      :[d(1.5)]*+{G(B)(G(B')GF(X))}="dr" _<<<<<<<{G(x)(G(B')GF(X))}
      :"rd" _\psi
    )
  ,
    :[d(1.05)r(.2)]*+{GF(A)G(F(A')F(X))} ^{GF(A)G(a'_X)}
    ( :[d(1.05)r(.2)]*+{GF(A)G(B'F(X))} ^{GF(A)G(x'F(X))}
       ( :"mr" ^{GF(A)b'_{F(X)}}
       ,
         :[d(1.9)l(.4)]*+{G(B)G(B'F(X))}="dm" ^{G(x)G(B'F(X))}
         :"dr" _{Bb'_{F(X)}}
       )
    ,
      :[d(1.2)l(.4)]*+{G(B)G(F(A')F(X))}="m" ^>>>>>>>>>>>>{G(x)G(F(A')F(X))}
      :"dm" _>>>>>>{G(B)G(x'F(X))}
    )
  ,
    :[d(1.2)l(.4)]*+{G(B)GF(A'X)}="um" _<<<<<<<<<<{G(x)GF(A'X)}
    :"m" _>>>>>>{G(B)G(a'_X)}
  )
,
  :[d(1)l(.3)]*+{G(F(A)F(A'X))} _{G(a_{A'X})}
  ( :[d(1.1)r(.3)]*+{G(BF(A'X))} ^{G(xF(A'X))}
    ( :"um" _{b_{F(A'X)}}
    ,
      :[d(.9)l(.3)]*+{G(B(F(A')F(X)))}="l" ^>>>>>>>{G(Ba'_X)}
      ( :"m" _>>>>>>>>>{b_{F(A')F(X)}}
      ,
        :[d(1)r(.3)]*+{G(B(B'F(X)))} _{G(B(x'F(X)))}
        ( :"ld" _{G(\phi)}
        ,
          :"dm" _{b_{B'F(X)}}
        )
      )
    )
  ,
    :[d(.9)l(.3)]*+{G(F(A)(F(A')F(X)))} _{G(F(A)a'_X)}
    ( :"lm" _<<<<<<{G(\phi)} |(.7){\hole}
    ,
      :"l" _{G(x(F(A')F(X)))}
    )
  )
)
}$$

\end{proof}

\subsection{Subcategories of multiplicants}

Let $\cG,\ \caH,\ \cC$ be monoidal categories equipped with monoidal functors $F':\cG\to \cC,\ F'':\caH\to \cC$. Let
$F:\cG\to \caH$ be a functor and let $\gamma:F''\circ F\to F'$ be an isomorphism of functors.
Define a full subcategory $\cM_l(\gamma) = \cM_l(F,F',F'',\gamma)$ in $\cM_l(F)$ of objects $(A,a)$ such that the following diagram is commutative

\xymatrix{ F'(A\otimes X) \ar[dd]_{F'_{A,X}} & F''F(A\otimes X) \ar[l]_{\gamma_{A\otimes X}} \ar[d]^{F''(a_X)} \\ & F''(F(A)\otimes F(X)) \ar[d]^{F''_{A,X}} \\ F'(A)\otimes F'(X)& F''F(A)\otimes F''F(X) \ar[l]_{\gamma_A\otimes\gamma_X} }

\begin{theo}
The subcategory $\cM_l(\gamma)$ is closed under the tensor product in $\cM_l(F)$. The natural isomorphism $\gamma$ lifts to a monoidal natural isomorphism $\overline\gamma$ which fills the square
\xymatrix{ \cM_l(\gamma) \ar[r] \ar[d] & \caH \ar[d]^{F''} \\ \cG \ar[r]^{F'} & \cC }
\end{theo}
\begin{proof}
The fact that the tensor product $(A\otimes B,a|b)$ of two objects $(A,a),\ (B,b)$ of $\cM_l(\gamma)$ belongs to $\cM_l(\gamma)$ follows from the commutative diagram:

$$
\xygraph{ !{0;/r12.5pc/:;/u6.5pc/::}[]*+{F''F(A(B(X))}
(
 :[d(3)]*+{F'(A(BX))} _{\gamma_{A(BX)}}
 ( :[u(0.6)r(0.4)]*+{F'((AB)X)}="dl" ^{F'(\phi)}
   :[r(0.7)]*+{F'(AB)F'(X)}="dm" ^{F'_{AB,X}} |(.42){\hole}
   :[r(1)]*+{(F'(A)F'(B))F'(X)}="dr" ^{F'_{A,B}F'(X)} |(.6){\hole}
 ,
   :[r(0.7)]*+{F'(A)F'(BX)}="d1" _{F'_{A,BX}}
   :[r(1)]*+{F'(A)(F'(B)F'(X))}="d2" _{F'(A)F'_{B,X}}
   :"dr" _{\psi}
 )
 ,
 :[u(0.6)r(0.4)]*+{F''F((AB)X)} ^{F''F(\phi)}
 ( :"dl" _{\gamma_{(AB)X}} |(.2){\hole}
 ,
   :[r(0.7)]*+{F''(F(AB)F(X))} ^{F''((a|b)_X)}
   (
     :[r(1)]*+{F''((F(A)F(B))F(X))}="r" ^{F''(a_BF(X))}
     :[d(1)]*+{F''(F(A)F(B))F''F(X)}="ru" ^{F''_{F(A)F(B),F(X)}}
     :[d(1)]*+{(F''F(A)F''F(B))F''F(X)}="rm" ^{F''_{F(A),F(B)}F(X)}
     :"dr" ^{(\gamma_A\gamma_B)\gamma_X}
   ,
     :[d(1)]*+{F''(F(AB))F''F(X)} _<<<<{F''_{F(AB),F(X)}} |(.6){\hole}
     ( :"ru" ^{F''(a_B)F''F(X)}
     ,
       :"dm" _{\gamma_{AB}\gamma_X} |(.3){\hole}
     )
   )
 )
 ,
 :[r(0.7)]*+{F''(F(A)F(BX))} _{F''(a_{BX})}
 ( :[r]*+{F''(F(A)(F(B)F(X))}  ^{F''(F(X)b_X}
   ( :"r" ^{F''(\phi')}
   ,
     :[d]*+{F''F(A)F''(F(B)F(X))}="m" ^{F''_{F(A),F(B)F(X)}}
     :[d]*+{F''F(A)(F''F(B)F''F(X))} _{F''F(A)F''_{F(B),F(X)}}
     ( :"rm" _{\psi}
     ,
       :"d2" _>>>>>>{\gamma_A(\gamma_B\gamma_X)}
     )
   )
 ,
   :[d]*+{F''F(A)F''F(BX)} _{F''_{F(A),F(BX)}}
   ( :"m" _{F''F(A)F''(b_X)}
   ,
     :"d1" _{\gamma_A\gamma_{BX}}
   )
 )
)
}$$

It follows from the definition of $\cM_l(\gamma)$ that the natural transformation $\overline\gamma$ given by $\overline\gamma_(A,a)=\gamma_A:F''F(A)\to F'(A)$ is monoidal.
\end{proof}

\section{Examples, categories of modules}

Here $k$ denote the ground field.

\subsection{Nuclei and multiplicants of categories of modules}

We start with the following technical statement.
\begin{lem}\label{enm}
Let $M$ be a left module over an associative algebra $R$ and $m_{X_1,...,X_n}\in End(M\otimes X_1\otimes
...\otimes X_n)$ be a family of linear operators, which is
natural in left $R$-modules $X_i$. Then $m_{X_1,...,X_n}$ is given by multiplication by an element $m\in End(M)\otimes R\otimes
...\otimes R$. Families of invertible operators correspond to invertible elements in the algebra $End(M)\otimes
R\otimes ...\otimes R$.
\end{lem}
\begin{proof}
The $R$-module $R$ is a generator in $R$-$Mod$. Thus any transformation $m_{X_1,...,X_n}$ natural in $X_i\in R$-$Mod$ is defined by its specialisation $m_{R,...,R}\in End(M\otimes R\otimes
...\otimes R)$. Naturality of $m$ implies that $m_{R,...,R}$ commutes with $1\otimes End_R(R)\otimes...\otimes End_R(R)$. Thus $m_{R,...,R}$ is given by multiplication by an element $m\in End(M)\otimes R\otimes
...\otimes R$.
\end{proof}

\begin{exam}Nucleus of a category of modules.
\end{exam}
Let $R$ be an algebra and $\Delta:R\otimes R\to R$ a homomorphism of algebras defining a magmoidal structure on
the category $R$-$Mod$ of left $R$-modules. By lemma \ref{enm} the nucleus $\cN_l(R$-$Mod)$ is the category of
pairs $(M,m)$ where $M$ is a left $R$-module and $m$ is an invertible element of the tensor product $End(M)\otimes
R\otimes R$ satisfying
\begin{equation}\label{inv}
(\Delta\otimes I)\Delta(r)m = m(I\otimes\Delta)\Delta(r)
\end{equation}
for all $r\in R$. Here $End(M)$ is the endomorphism ring of the vector space $M$. Tensor product on the category
$\cN_l(R$-$Mod)$ is defined by $(M,m)\otimes(N,n) = (M\otimes N,m|n)$ where $m|n\in End(M)\otimes End(N)\otimes
R\otimes R$ is given by
\begin{equation}\label{tenpr}
(m|n) = (m\otimes 1)(I\otimes\Delta\otimes I)(m)(1\otimes n)(I\otimes I\otimes\Delta)(m)^{-1}.
\end{equation}
The tensor product is semigroupal with the associativity constraint: $$\phi_{(M,m),(N,n),(L,l)} = (I\otimes\rho_N\otimes\rho_L)(m).$$ Here $\rho_N:R\to End(N),\ \rho_L:R\to End(L)$ are corresponding $R$-module structures.

If $\varepsilon:R\to k$ is a homomorphism of algebras ({\em counit}) such that $$(\varepsilon\otimes I)\Delta = I = (I\otimes\varepsilon)\Delta,$$ the magmoidal category $R$-$Mod$ has a unit object $(k,1)$. The monoidal subcategory $\cN^1_l(R$-$Mod)$ consists of those pairs $(M,m)$, where $m$ satisfies to the {\em normalisation} condition: $$(I\otimes\varepsilon\otimes I)(m) = 1 = (I\otimes I\otimes\varepsilon)(m).$$

In the case when $\Delta$ is coassociative (when $(R,\Delta)$ is a bialgebra) the equation (\ref{inv}) means that
$m$ is $R$-invariant with respect to the diagonal inclusion of $R$ into $End(M)\otimes R\otimes R$.
The monoidal functor $R$-$Mod\to \cN_l(R$-$Mod)$ corresponding to the canonical (trivial) associativity constraint on $R$-$Mod$, sends an $R$-module $M$ into a pair $(M,1)$.

\begin{exam}Multiplicant of a homomorphism.
\end{exam}
Let $H_1, H_2$ be bialgebras and let $f:H_1\to H_2$ be a homomorphism of algebras. It defines a functor
$f^*:H_2$-$Mod\to H_1$-$Mod$ between monoidal categories. Here we describe its (left) multiplicant $\cM_l(f^*)$.
Objects of $\cM_l(f^*)$ are pairs $(M,m)$ where $M$ is $H_2$-module and $m$ is an invertible element of
$End(M)\otimes H_2$ satisfying
\begin{equation}\label{multbialg}
m\Delta(f(h)) = (f\otimes f)\Delta(h)m
\end{equation}
for all $h\in H_1$. Tensor product of
pairs is given by the formula $(M,m)\otimes (N,n) = (M\otimes N,m|n)$ where $m|n\in End(M)\otimes End(N)\otimes
H_2$ is defined by $$(m|n) = (m\otimes 1)^{-1}(1\otimes n)(I\otimes\Delta)(m).$$

The forgetful functor $\cM_l(f^*)\to H_2$-$Mod$ sends a pair $(M,m)$ to $M$. The functor $f^*:\cM_l(f^*)\to H_1$-$Mod$, which sends $(M,m)$ into $f^*(M)$ is semigroupal with respect to the constraint: $$f^*_{(M,m),(N,n)} = (I\otimes\rho_N)(m)\in End(M)\otimes End(N).$$

Let $F_i:H_i-Mod\to \Vect$ be the forgetful functors. An isomorphism $\gamma:F_1\to F_2f^*$ is given by multiplication with an invertible $x\in H_2$. The subcategory $\cM_l(\gamma)$ of the multiplicant $\cM_l(f^*)$ consists of pair $(M,m)$ such that $$m = (\rho_M\otimes I)(\Delta(x)^{-1}(x\otimes x)),$$ where $\rho_M:H_2\to End(M)$ is the $H_2$-action on $M$. Note that this together with (\ref{multbialg}) is equivalent to the condition: $$(\rho_M\otimes I)(\Delta(g(h))) = (\rho_M\otimes I)((g\otimes g)\Delta(h)),\quad h\in H_1,$$ where $g:H_1\to H_2$ is given by $g(h) = xf(h)x^{-1}$.

\subsection{Algebraic constructions}

For an associative algebra $R$ consider the category $\E(R)$ of pairs $(V,v)$ where $V$ is a vector space and $v$ is
an element of $End(V)\otimes R$ with morphisms being vector space maps compatible with second components: a
morphism $(V,v)\to(U,u)$ is a map $f:V\to U$ such that the equality
\begin{equation}\label{coco}
u(f\otimes 1) = (f\otimes 1)v
\end{equation}
is valid in $Hom(V,U)\otimes R$. Denote by $\A(R)$ the full
subcategory of $\E(R)$ of those pairs $(V,v)$ for which $v$ is invertible in $End(V)\otimes R$.
\begin{lem}
Let $R$ be a finite dimensional associative algebra. Then the category $\E(R)$ is equivalent to the category of
modules over the free algebra $T(R^\sve)$ generated by the dual space of $R$. The category $\A(R)$ is equivalent
to the category of modules over the quotient of the free algebra $T(R^\sve\oplus R^\sve)$ by the ideal generated
by $\{ \delta'(l)-l(1), \delta"(l)-l(1); l\in R^\sve\}$ where $\delta',\delta"$ are the compositions of
$\delta:R^\sve\to R^\sve\otimes R^\sve$, which is the dual map to the multiplication $R\otimes R\to R$, with two
inclusions of $R^\sve\otimes R^\sve$ into $(R^\sve\oplus R^\sve)^{\otimes 2}$.
\end{lem}
\begin{proof}
An element $v$ of $End(V)\otimes R$ corresponds to a map $f_v:R^\sve\to End(V)$ sending $l\in R^\sve$ into $(I\otimes
l)(v)$, and thus a $T(R^\sve)$-module structure on $V$. Clearly compatibility condition (\ref{coco}) guarantees
that a morphisms of pairs is $T(R^\sve)$-linear. In the case when $v$ is invertible the maps $f_v,f_{v^{-1}}$
satisfy to the conditions: $$(f_v\otimes f_{v^{-1}})\delta(l) = (I\otimes\delta(l)(v_{12}v^{-1}_{13}) = (I\otimes
l)(vv^{-1}) = l(1),$$ $$(f_{v^{-1}}\otimes f_v)\delta(l) = (I\otimes\delta(l)(v^{-1}_{12}v_{13}) = (I\otimes
l)(v^{-1}v) = l(1).$$
\end{proof}

Let $R = k(G)$ be an algebra of functions on a finite group $G$ with standard multiplication $\Delta$ given by
$\Delta(l)(f\otimes g) = l(fg)$ for $l\in k(G)$ and $f,g\in G$. By $p_g\in k(G)$ we denote the $\delta$-function
concentrated at $g\in G$. Note that $\delta$-functions $p_g, g\in G$ form a basis in $k(G)$ so we can write $m\in
End(M)\otimes k(G)\otimes k(G)$ as $\sum_{f,g\in G}m(f,g)\otimes p_f\otimes p_g$. The element $m$ is invertible if
$m(f,g)\in End(M)$ are invertible for any $f,g$. $k(G)$-invariance of $m$ (\ref{inv}) means that $m(f,g)$ commutes
with the image of $k(G)$ in $End(M)$. Thus we can identify the nucleus $\cN_l(k(G)$-$Mod)$ with the category
$N_l(k(G))$-$Mod$ of left modules over the algebra $N_l(k(G)) = k(G)\otimes k[F(G\times G)]$ which is a tensor
product of $k(G)$ and the group algebra of the free group $F(G\times G)$ generated by the cartesian product
$G\times G$. Generators of $F(G\times G)$ will be denoted by $u(f,g)$. The pair $(M,m)$ corresponds to the
$N_l(k(G))$ action on $M$ sending $u(f,g)$ into $m(f,g)$. The formula (\ref{tenpr}) for the tensor product of two
pairs gives $$(m|n)(f,g) = \sum_{h\in G}m(h,f)m(hf,g)m(h,fg)^{-1}\otimes p_hn(f,g)$$ which corresponds to the
following comultiplication on $N_l(k(G))$ $$\Delta(u(f,g)) = \sum_{h\in G}u(h,f)u(hf,g)u(h,fg)^{-1}\otimes
p_hu(f,g).$$ Note that $k(G)$ is a sub-bialgebra of $N_l(k(G))$. The coproduct on $N_l(k(G))$ is not coassociative
but rather quasi-associative in the sense of \cite{dr}: $$(I\otimes\Delta)(x) = \Phi(\Delta\otimes
I)\Delta(x)\Phi^{-1}$$ where $\Phi = \sum_{f,g\in G}u(f,g)\otimes p_f\otimes p_g$ is an invertible element of
$N_l(k(G))^{\otimes 3}$ ({\em associator}).

It is well known that associativity constraints on the category $k(G)-Mod$ are in one-to-one correspondence with
3-cocycles $Z^3(G,k^*)$ of the group $G$ with coefficients in invertible elements of the ground field $k$. For any
3-cocycle $\alpha\in Z^3(G,k^*)$ there is defined a homomorphism of quasi-bialgebras $N_l(k(G))\to k(G)$ splitting
the inclusion $k(G)\to N_l(k(G))$ which sends $u(f,g)$ into $\sum_{h\in G}\alpha(h,f,g)p_h$.

\end{document}